\documentclass{amsart}
\usepackage[svgnames]{xcolor}
\usepackage[unicode,
  colorlinks=true,
  linktocpage=true,
  citecolor=ForestGreen,
  linkcolor=MediumOrchid,
  urlcolor=MediumOrchid,
  pdftitle={Models for cyclic infinity operads},
  pdfauthor={Brandon Doherty and Philip Hackney}
  ]{hyperref}
\usepackage{tikz-cd}
\usepackage{amssymb}
\usepackage{mathtools}
\usepackage{microtype}
\usepackage[capitalise,noabbrev]{cleveref}
\usepackage{mathrsfs}
\usepackage{enumitem}
\usepackage[T1]{fontenc}

\usetikzlibrary{backgrounds}
\usetikzlibrary{arrows}
\usetikzlibrary{shapes,shapes.geometric,shapes.misc}

\tikzstyle{tikzfig}=[baseline=-0.25em,scale=0.5]

\pgfkeys{/tikz/tikzit fill/.initial=0}
\pgfkeys{/tikz/tikzit draw/.initial=0}
\pgfkeys{/tikz/tikzit shape/.initial=0}
\pgfkeys{/tikz/tikzit category/.initial=0}

\pgfdeclarelayer{edgelayer}
\pgfdeclarelayer{nodelayer}
\pgfsetlayers{background,edgelayer,nodelayer,main}

\tikzstyle{none}=[inner sep=0mm]

\newcommand{\tikzfig}[1]{%
{\tikzstyle{every picture}=[tikzfig]
\IfFileExists{#1.tikz}
  {\input{#1.tikz}}
  {%
    \IfFileExists{./figures/#1.tikz}
      {\input{./figures/#1.tikz}}
      {\tikz[baseline=-0.5em]{\node[draw=red,font=\color{red},fill=red!10!white] {\textit{#1}};}}%
  }}%
}

\tikzstyle{medium box}=[fill=white, draw=black, shape=rectangle, minimum width=0.75cm, minimum height=1cm]
\tikzstyle{short box}=[fill=white, draw=black, shape=rectangle, minimum width=0.75cm, minimum height=0.5cm]

\tikzstyle{blue edge}=[-, draw=blue]

\DeclareSymbolFont{symbolsC}{U}{txsyc}{m}{n}
\SetSymbolFont{symbolsC}{bold}{U}{txsyc}{bx}{n}
\DeclareFontSubstitution{U}{txsyc}{m}{n}
\DeclareMathSymbol{\Nearrow}{\mathrel}{symbolsC}{116}
\DeclareMathSymbol{\Searrow}{\mathrel}{symbolsC}{117}
\DeclareMathSymbol{\Nwarrow}{\mathrel}{symbolsC}{118}
\DeclareMathSymbol{\Swarrow}{\mathrel}{symbolsC}{119}

\newtheorem{theorem}{Theorem}[section]
\newtheorem{thmx}{Theorem}

\newtheorem{proposition}[theorem]{Proposition}
\newtheorem{corollary}[theorem]{Corollary}
\newtheorem{lemma}[theorem]{Lemma}
\theoremstyle{definition}
\newtheorem{definition}[theorem]{Definition}
\newtheorem{example}[theorem]{Example}
\newtheorem{convention}[theorem]{Convention}
\newtheorem{exercise}[theorem]{Exercise}
\theoremstyle{remark}
\newtheorem{remark}[theorem]{Remark}

\DeclareFontFamily{U} {MnSymbolC}{}

\DeclareFontShape{U}{MnSymbolC}{m}{n}{
  <-6> MnSymbolC5
  <6-7> MnSymbolC6
  <7-8> MnSymbolC7
  <8-9> MnSymbolC8
  <9-10> MnSymbolC9
  <10-12> MnSymbolC10
  <12-> MnSymbolC12}{}
\DeclareFontShape{U}{MnSymbolC}{b}{n}{
  <-6> MnSymbolC-Bold5
  <6-7> MnSymbolC-Bold6
  <7-8> MnSymbolC-Bold7
  <8-9> MnSymbolC-Bold8
  <9-10> MnSymbolC-Bold9
  <10-12> MnSymbolC-Bold10
  <12-> MnSymbolC-Bold12}{}

\DeclareSymbolFont{MnSyC} {U} {MnSymbolC}{m}{n}
\DeclareMathSymbol{\medstar}{\mathbin}{MnSyC}{130}

\newcommand{\mydef}[1]{\emph{\color{DarkOrchid}#1}}

\newcommand{\op}{\textup{op}}

\DeclareMathOperator{\ob}{Ob}
\DeclareMathOperator{\mor}{Mor}

\newcommand{\set}{\mathsf{Set}}
\newcommand{\fun}{\mathsf{Fun}}
\DeclareMathOperator{\map}{map}

\newcommand{\freeiso}{\mathbb{J}}

\newcommand{\rigc}{\mathfrak{C}}
\newcommand{\hNerve}{\mathfrak{N}}
\newcommand{\hcn}{\hNerve}
\newcommand{\unda}{\underline{a}}
\newcommand{\Pbd}[1]{P_{\perim}(#1)}

\newcommand{\simpcat}{\mathsf{\Delta}}
\newcommand{\dendcat}{\mathsf{\Omega}}
\newcommand{\cdcat}{\mathsf{\Upsilon}}
\newcommand{\cdcatrm}{\Upsilon}
\newcommand{\cscat}{\raisebox{\depth}{\scalebox{-1}[-1]{$\mathsf{\Delta}$}}}

\newcommand{\repo}[1]{\Omega^{#1}}
\newcommand{\repu}[1]{\cdcatrm^{#1}}

\newcommand{\segcore}[1]{\mathrm{Sc}^{#1}}

\newcommand{\actrm}{\textup{act}}
\newcommand{\intrm}{\textup{int}}

\newcommand{\catM}{\mathcal{M}}
\newcommand{\catN}{\mathcal{N}}

\newcommand{\cC}{\mathcal{C}}

\newcommand{\cycop}{\mathsf{Cyc}}
\newcommand{\cyc}{\cycop}
\newcommand{\operad}{\mathsf{Op}}
\newcommand{\cat}{\mathsf{Cat}}
\newcommand{\scycop}{\mathsf{sCyc}}
\newcommand{\scyc}{\scycop}
\newcommand{\soperad}{\mathsf{sOp}}
\newcommand{\scat}{\mathsf{sCat}}

\newcommand{\plcyc}{\mathsf{plCyc}}
\newcommand{\plop}{\mathsf{plOp}}
\newcommand{\splop}{\mathsf{splOp}}
\newcommand{\splcyc}{\mathsf{splCyc}}

\newcommand{\plrtree}{\dendcat_\textup{p}}
\newcommand{\pltree}{\cdcat_\textup{p}}

\newcommand{\pdset}{\dset_\textup{p}}
\newcommand{\cpdset}{\cdset_\textup{p}}
\newcommand{\pdspace}{\dspace_\textup{p}}
\newcommand{\cpdspace}{\cdspace_\textup{p}}

\newcommand{\rooting}{\mathfrak{r}}
\newcommand{\planing}{\mathfrak{p}}

\newcommand{\ass}{\mathfrak{Ass}}
\newcommand{\comm}{\mathfrak{Com}}

\newcommand{\pre}[1]{\widehat{#1}}
\newcommand{\spshf}[1]{\textup{s}\widehat{#1}}
\newcommand{\dset}{\pre{\dendcat}}
\newcommand{\dspace}{\spshf{\dendcat}}
\newcommand{\cdset}{\pre{\cdcat}}
\newcommand{\cdspace}{\spshf{\cdcat}}
\newcommand{\sset}{\pre{\simpcat}}
\newcommand{\isset}{\pre{\cscat}}

\newcommand{\rezko}{\dspace_\textup{Rezk}}
\newcommand{\rezku}{\cdspace_\textup{Rezk}}
\newcommand{\prezko}{(\pdspace)_\textup{Rezk}}
\newcommand{\prezku}{(\cpdspace)_\textup{Rezk}}

\newcommand{\nbhd}{\mathtt{nb}}
\newcommand{\inp}{\mathtt{in}}
\newcommand{\out}{\mathtt{out}}
\newcommand{\perim}{\mathop{\mathtt{bd}}}
\newcommand{\subtree}{\mathtt{Sb}}

\DeclareMathOperator*{\colim}{colim}
\DeclareMathOperator{\id}{id}
\DeclareMathOperator{\dom}{dom}
\DeclareMathOperator{\cod}{cod}
\DeclareMathOperator{\iso}{Iso}

\newcommand{\reedyS}{\mathbb{S}}
\newcommand{\reedyR}{\mathbb{R}}
\newcommand{\reedyG}{\mathbb{G}}

\newcommand{\comp}[2]{\circ_{#1}^{#2}}
\newcommand{\ua}{\underline{a}}
\newcommand{\uc}{\underline{c}}
\newcommand{\ud}{\underline{d}}

\newcommand{\sym}{\operatorname{Sym}}

\newcommand{\linear}[1]{\langle{#1}\rangle}

\DeclareMathOperator{\colors}{Col}
\DeclareMathOperator{\aut}{Aut}

\subjclass[2020]
{18M85, 
18N40, 
55U35, 
18N70} 

\keywords{cyclic operad, $\infty$-operad, dendroidal set}

\begin{document}

\title{Models for cyclic infinity operads}
\author{Brandon Doherty}
\address{Department of Mathematics, Florida State University}
\email{bdoherty@fsu.edu}
\author{Philip Hackney}
\address{Department of Mathematics, University of Louisiana at Lafayette}
\email{philip@phck.net} 
\urladdr{http://phck.net}
\date{July 14, 2025}

\begin{abstract}
We construct model structures on cyclic dendroidal sets and cyclic dendroidal spaces for cyclic quasi-operads and complete cyclic dendroidal Segal spaces, respectively.
We show these models are Quillen equivalent to the model structure for simplicial cyclic operads.
This answers in the affirmative a question of the second author and Drummond-Cole concerning model structures for cyclic $\infty$-operads.
We infer similar statements for planar cyclic $\infty$-operads, providing the model-categorical foundation needed to complete Walde's program on the relationship between cyclic 2-Segal spaces and planar cyclic $\infty$-operads.
\end{abstract}

\thanks{
This material is partially based upon work supported by the National Science Foundation under Grant No. DMS-1928930 while the authors participated in a program supported by the Mathematical Sciences Research Institute; 
the program was held in the summer of 2022 in partnership with the Universidad Nacional Autónoma de México.
The first author was supported in part by a grant from the Knut and Alice Wallenberg Foundation, entitled ``Type Theory for Mathematics and Computer Science'' (principal investigator: Thierry Coquand).
The second author was partially supported by the Louisiana Board of Regents through the Board of Regents Support fund, contract number LEQSF(2024-27)-RD-A-31. 
This work was supported by a grant from the Simons Foundation (PH \#850849).
}

\maketitle

\section{Introduction}

Cyclic operads were introduced by Getzler and Kapranov in order to generalize Connes--Tsygan cyclic homology for associative algebras to more exotic algebraic structures \cite{GetzlerKapranov:COCH}.
Cyclic operads have since found applications ranging from Kontsevich's graph complexes to topological field theory.
Meanwhile, the dendroidal objects of Moerdijk and Weiss \cite{MoerdijkWeiss:DS,MoerdijkWeiss:OIKCDS} play a prominent role in the theory of $\infty$-operads. 
Several model category structures based on dendroidal objects were developed and compared by Cisinski and Moerdijk, paralleling earlier developments in the theory of $(\infty,1)$-categories by many others.
The dendroidal models of Cisinski and Moerdijk are equivalent to approaches to $\infty$-operads given by Barwick and Lurie \cite{Barwick:OCHO,ChuHaugsengHeuts:2MHTIO,HeutsHinichMoerdijk:OEBLDMIO}.
This paper is concerned with \emph{cyclic dendroidal objects}.

A \mydef{cyclic quasi-operad} is a cyclic dendroidal set $X$ which has fillers for all inner horns (where $e$ is an inner edge of an unrooted tree $T$):
\[ \begin{tikzcd}
\Lambda^T_e \rar \dar & X \\
\repu{T} \ar[ur, dashed,"\exists"']
\end{tikzcd} \]
This diagram is to be interpreted in the category of cyclic dendroidal sets, an analogue of the dendroidal sets of Moerdijk--Weiss \cite{MoerdijkWeiss:DS,MoerdijkWeiss:OIKCDS}, but based on the \mydef{cyclic dendroidal category $\cdcat$} introduced by the second author, Robertson, and Yau \cite{HackneyRobertsonYau:GCHMO}.

\begin{thmx}
There is a Quillen model structure on the category of cyclic dendroidal sets whose fibrant objects are the cyclic quasi-operads and whose cofibrations are the normal monomorphisms.
A morphism is a weak equivalence just when the underlying map of dendroidal sets is a weak operadic equivalence.
\end{thmx}

This is analogous to Joyal's model structure for quasi-categories.
In fact, there is a subcategory whose fibrant objects are the \emph{anti-involutive} quasi-categories, previously studied in \cite{DrummondColeHackney:CERIMS}.

There is a homotopy coherent nerve functor $\hcn \colon \scyc \to \cdset$ from the category of simplicial cyclic operads to the category of cyclic dendroidal sets. 
There is a Dwyer--Kan style model structure on $\scyc$ \cite{DrummondColeHackney:DKHTCO} (analogous to Bergner's model structure for simplicial categories \cite{Bergner:MCSC} and Cisinski--Moerdijk model structure for simplicial operads \cite{CisinskiMoerdijk:DSSO}).
We establish the following rigidification result for cyclic quasi-operads.
\begin{thmx}\label{intro thm rigidification}
The homotopy coherent nerve/rigidification adjunction is a Quillen equivalence between the model structure for cyclic quasi-operads and the model structure on simplicial cyclic operads.
\end{thmx}

Cyclic dendroidal spaces are presheaves on $\cdcat$ valued in simplicial sets.
A (Reedy fibrant) cyclic dendroidal space $X$ is said to be \emph{Segal} if for each tree $T$, the map
\[
  X_T \to \lim_{U \subset T} X_U
\]
is a weak homotopy equivalence of simplicial sets; the limit is taken over the collection of edges and vertices $U$ of the tree $T$.
Such cyclic dendroidal spaces admit weakly defined $m$-ary cyclic-operad-style compositions (see Figure~\ref{fig cyc comp}), via the following span diagram (where $T$ is a tree with $m$ vertices and $n$ boundary edges):
\[ \begin{tikzcd}[sep=small]
\prod\limits_{v\in T} X_{\medstar_{|v|}} & 
\lim\limits_{U \subset T} X_U \lar[hook']
&
X_T \rar \lar["\simeq"'] &
X_{\medstar_n}. 
\end{tikzcd} \]
A cyclic dendroidal Segal space is said to be \emph{complete} if its underlying simplicial space is a complete Segal space in the sense of Rezk \cite{Rezk:MHTHT}.

\begin{figure}
  
\begin{tikzpicture}
    \node[inner sep=0pt] (bg) at (0,0) {\includegraphics[width=288pt]{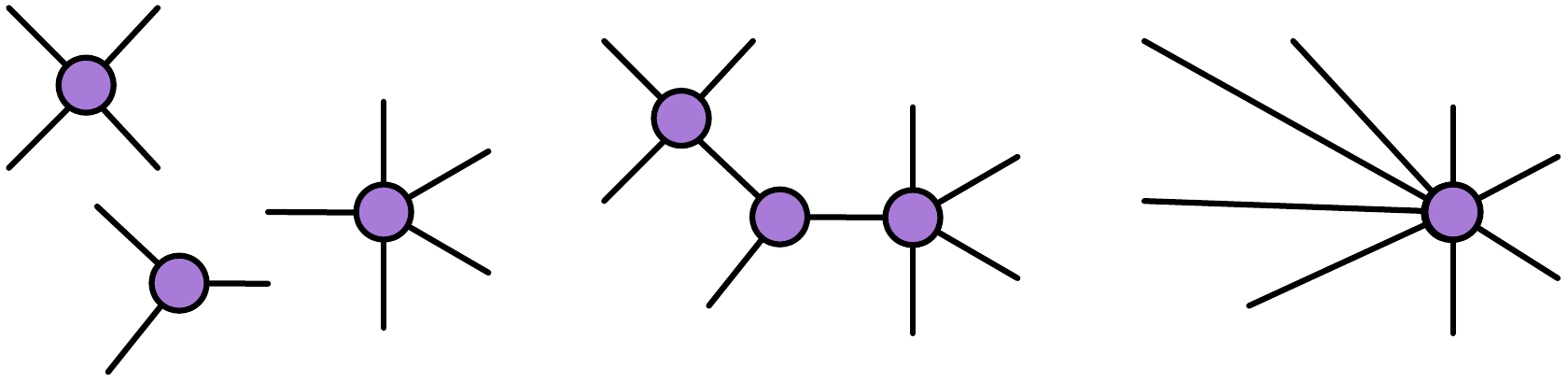}};
    
    \node at (-4.1, 0.7) {$f$};
    \node at (-4.3, -0.6) {$g$};
    \node at (-2.9, 0.2) {$h$};
    
    \node at (-0.65, 0.9) {$f$};
    \node at (0.2, 0.2) {$g$};
    \node at (1.3, -0.15) {$h$};
    
    \node at (5.4, -0.15) {$f \circ g \circ h$};
    
    \node at (-1.6, -0.15) {$\rightsquigarrow$};
    
    \node at (1.9, -0.15) {$\rightsquigarrow$};
    
\end{tikzpicture}
\caption{Ternary composition in a cyclic operad}\label{fig cyc comp}
\end{figure}

\begin{thmx}\label{intro thm Rezk model}
There is a Quillen model structure on the category of cyclic dendroidal spaces called the \mydef{cyclic dendroidal Rezk model structure} with fibrant objects the complete cyclic dendroidal Segal spaces and with cofibrations the Reedy cofibrations. 
\end{thmx}

\begin{thmx}\label{intro thm cdset cdspace}
The inclusion of sets into simplicial sets underlies
a left Quillen equivalence from the model structure for cyclic quasi-operads to the cyclic dendroidal Rezk model structure.
\end{thmx}

These results resolve a conjecture of the second author and Drummond-Cole \cite[Remark 6.8]{DrummondColeHackney:DKHTCO} (see also \cite[Question 6.5]{Hackney:SCGO}) by leveraging the corresponding results for $\infty$-operads proved by Cisinski and Moerdijk \cite{CisinskiMoerdijk:DSMHO,CisinskiMoerdijk:DSSIO,CisinskiMoerdijk:DSSO}.

\subsection*{Planar cyclic \texorpdfstring{$\infty$}{infinity}-operads}
There are parallel results to those above in the planar/nonsymmetric case.
We prove the following omnibus theorem.

\begin{thmx}\label{intro thm planar}
The homotopy coherent nerve functor is a right Quillen equivalence from the model structure for planar (cyclic) operads to the model structure for planar (cyclic) quasi-operads, and the inclusion from the model structure for planar (cyclic) quasi-operads to the planar (cyclic) dendroidal Rezk model structure is a left Quillen equivalence.
\end{thmx}

Apparently the non-cyclic versions do not yet appear in the literature. 
We prove this theorem using slicing techniques: planar (cyclic) operads are the same thing as (cyclic) operads sliced over the associative (cyclic) operad, and similarly for planar (cyclic) dendroidal sets/spaces.

Our interest in the planar case is primarily motivated by a connection to (cyclic) 2-Segal spaces investigated by Walde \cite{Walde:2SSIIO};
we focus here on the cyclic objects.
Walde constructs a functor $\pltree \to \mathsf{\Lambda}$ landing in Connes' cyclic category \cite{Connes:CCFE} sending a planar tree to the cyclically-ordered set of gaps between its external edges.
Any cyclic space which satisfies the 2-Segal condition \cite{DyckerhoffKapranov:TSTC,DyckerhoffKapranov:HSS} is sent by restriction along this functor to a planar complete Segal cyclic dendroidal space.
At the level of $\infty$-categories, this identifies cyclic 2-Segal spaces with the \emph{invertible} planar cyclic $\infty$-operads.
Remark 5.0.20 of \cite{Walde:2SSIIO} explains how a positive resolution of the planar version of the Drummond-Cole--Hackney conjecture would make the preceding sentence precisely true.
With the establishment of the above theorem, we consider this matter resolved.

\subsection*{Related work and future directions}

\begin{enumerate}[label=\arabic*., ref=\arabic*, left=0pt]
\item 
Barkan and Steinebrunner are currently developing $\infty$-categorical models for cyclic operads \cite{Steinebrunner:2DTFT,BarkanSteinebrunner:CIO},
which should compare favorably to the Segal-type model structures developed in this paper.
\item An important task is to construct suitable model structures for modular $\infty$-operads \cite{GetzlerKapranov:MO,HackneyRobertsonYau:GCHMO}.
We do not expect a rigidification result like \cref{intro thm rigidification} will hold, but there is hope for a modular version of \cref{intro thm cdset cdspace}.
\item Walker Stern is currently building a fibrational framework for planar cyclic $\infty$-operads, and exploring connections between these approaches would be valuable.
\item The covariant model structure on slices of dendroidal sets or spaces is used to model algebras over an $\infty$-operad \cite[\S9.5, Chapter 13]{HeutsMoerdijk:SDHT}.
Are there corresponding model structures for cyclic algebras for a cyclic $\infty$-operad?
\item Bonatto and Robertson \cite{BonattoRobertson:GAS} study Grothendieck--Teichmüller theory via (cyclic and modular) Segal presheaves in pro-groupoids/pro-spaces. 
This suggests developing profinite versions of model structures in this paper, or cyclic (and modular) analogues of the profinite $\infty$-operads of Blom and Moerdijk \cite{BlomMoerdijk:PIO,BlomMoerdijk:PCTO}.
\end{enumerate}

Finally, we note that the kind of colored cyclic operad we discuss in this paper is more general than what appears in \cite{HackneyRobertsonYau:HCO} as it allows for involutive sets of colors, and we work with a better category of trees (see \cref{rmk xi}).
The introductions to \cite{DrummondColeHackney:DKHTCO} and \cite{HackneyRobertsonYau:MONT} discuss advantages to working in the involutive setting. 

\subsection*{Guide to the paper}
The next two sections are strictly background on adjoint strings, lifted model structures, operads, and cyclic operads. 
\Cref{sec tree cats} collects important properties of the cyclic dendroidal category $\cdcat$ and its relationship with the dendroidal category $\dendcat$, and also provides foundational material concerning presheaves on these categories. 
Existing model structures, for quasi-operads and for anti-involutive quasi-categories, are recalled in \cref{section old models}, along with new information concerning the latter.

In \cref{sec cyclic quasi-operads} we establish the model structure for cyclic quasi-operads, which involves a careful analysis of the functors relating dendroidal and cyclic dendroidal sets.
The heart of the paper is \cref{sec QE simplicial cyclic}, where we study the homotopy coherent nerve/rigidification adjunction and prove \cref{intro thm rigidification}.
In \cref{sec rezk ms} we introduce the cyclic dendroidal Rezk model structure on cyclic dendroidal spaces, which admits two descriptions: it is both right-induced from the dendroidal Rezk model structure, and a Bousfield localization of the generalized Reedy model structure.

In \cref{sec planar} we turn to the planar case.
Each model structure from \cref{intro thm planar} can be considered as a slice model structure, allowing us to apply general results to lift Quillen equivalences.
We also give alternative descriptions for several of these model structures.

The appendix collects a number of general results about model categories which are used in the main body of the paper.
These especially involve the theory of lifted model structures and their interactions with other model categorical machinery, such as Quillen equivalences, left Bousfield localization, and generalized Reedy model structures.

\section{Adjoint strings and lifted model structures} \label{sec adj strings}
A main technical tool in this paper is the concept of an \emph{adjoint string}, by which we mean a functor $F \colon \catN \to \catM$ together with a specified left adjoint $L$ and a specified right adjoint $R$. 
We will write such adjoint strings as $L \dashv F \dashv R$.
As a primary example, if $f \colon C \to D$ is a functor between small categories, then we obtain an adjoint string $f_! \dashv f^* \dashv f_*$ between the associated categories of presheaves.
\[
\begin{tikzcd}[column sep=large]
    \pre{D}  \rar["f^*" description] 
    \rar[phantom, bend right=18, "\scriptscriptstyle\perp"]  
    \rar[phantom, bend left=18, "\scriptscriptstyle\perp"] 
    & \pre{C}
    \lar[bend right=30, "f_!" swap, start anchor = {[yshift=1.5ex]west}, end anchor = {[yshift=1.5ex]east}] 
    \lar[bend left=30, "f_*", start anchor = {[yshift=-1.5ex]west}, end anchor = {[yshift=-1.5ex]east}] 
\end{tikzcd}
\]
Here $\pre{C} = \fun(C^\op, \set)$ is the category of presheaves, $f^*$ is given by precomposition with $f^\op \colon C^\op \to D^\op$, and $f_!$ and $f_*$ are respectively given by left and right Kan extension along $f^\op$.

\begin{example}\label{ex color sets}
Consider the inclusion $\iota$ of the trivial group $\{e\}$ into the group with two elements $C_2$, both considered as categories.
A presheaf on $\{e\}$ is just a set, and a presheaf on $C_2$ is just a set $A$ equipped with an involution $a \mapsto a^\dagger$. 
The functor $\iota^*$ forgets involutions, and we can take $\iota_!(S) = S \amalg S$ together with the swapping involution and $\iota_*(S) = S \times S$ together with reflection.
Put another way, the left adjoint is defined by $\iota_!S = C_2 \times S$ and the right adjoint by $\iota_*S = S^{C_2},$ each with the evident action of $C_2$.
\end{example}

\begin{definition}
Suppose $\catM$ and $\catN$ are model categories and $F \colon \catN \to \catM$ is a functor.
We say that
\begin{itemize}[left=0pt]
\item The model structure on $\catN$ is \emph{right-induced} from that on $\catM$ if $F$ admits a left adjoint and creates fibrations and weak equivalences.
\item The model structure on $\catN$ is \emph{left-induced} from that on $\catM$ if $F$ admits a right adjoint and creates cofibrations and weak equivalences.
\end{itemize}
\end{definition}

Since a model structure is fully determined by the fibrations and weak equivalences, there is at most one right-induced model structure on the category $\catN$ relative to a functor $F \colon \catN \to \catM$, where $\catM$ is a model category.
Our usual starting point is a model structure on $\catM$ and a functor $F \colon \catN \to \catM$ from a category.
If the right-induced model structure on $\catM$ exists, we will denote it by $\catN_r$ (and likewise by $\catN_l$ for the left-induced model structure).

Our interest in adjoint strings stems primarily from the following basic lemma for lifted model structures \cite{DrummondColeHackney:CERIMS,HackneyRovelli:IMSHC, Shulman:UU}.
A proof of the left-induced statement is included in \cref{subsec left induced}.

\begin{lemma}\label{triple-induced-model-structures}
Consider adjoint functors $L\dashv F \dashv R$ 
\[
\begin{tikzcd}[column sep=large]
    \catN  \rar["F" description] 
    \rar[phantom, bend right=18, "\scriptscriptstyle\perp"]  
    \rar[phantom, bend left=18, "\scriptscriptstyle\perp"] 
    & \catM
    \lar[bend right=30, "L" swap, start anchor = {[yshift=1.5ex]west}, end anchor = {[yshift=1.5ex]east}] 
    \lar[bend left=30, "R", start anchor = {[yshift=-1.5ex]west}, end anchor = {[yshift=-1.5ex]east}] 
\end{tikzcd}
\]
where $\catM$ is a cofibrantly-generated model category and $\catN$ is bicomplete.
If $FL \dashv FR$ is a Quillen adjunction on $\catM$, then the right-induced model structure $\catN_r$ on $\catM$ exists.
If, additionally, $\catM$ and $\catN$ are locally presentable (i.e.\ $\catM$ is combinatorial), then the left-induced model structure $\catN_l$ on $\catM$ exists as well. \qed
\end{lemma}

In the context of this lemma, it is often possible to lift a Quillen equivalence $\catM \simeq \catM'$ to a Quillen equivalence $\catN \simeq \catN'$, and we will use this technique multiple times.
See \cite[Theorem 5.6]{DrummondColeHackney:CERIMS} and \cref{sec lifted QE} for precise statements.

\section{Operads and cyclic operads}\label{sec operads and cyclic}
We now give background on operads and on cyclic operads, following \cite{DrummondColeHackney:DKHTCO}.
Operads are the same thing as (small) symmetric multicategories, and cyclic operads in our sense are also known as cyclic symmetric multicategories \cite[Definition 7.4]{Shulman:2Chu}; see \cite{ChengGurskiRiehl:CMMAM} for the nonsymmetric case. 

The notation $[0,n] = \{ 0, 1, \dots, n \}$ and $[1,n] = \{1, 2, \dots, n\}$ will refer to intervals of integers, which could be empty.
We let $\Sigma_n = \aut([1,n])$ and $\Sigma_n^+ = \aut([0,n])$ be the symmetric groups (so $\Sigma_n^+ \cong \Sigma_{n+1}$).
Suppose $C$ is a set of \emph{colors}.
A \emph{profile} in $C$ is a (possibly empty) word in $C$, that is, a function $[1,n] \to C$ or $[0,n] \to C$. 
Notice that there are (right) $\Sigma_n$ or $\Sigma_n^+$ actions on the set of profiles of a fixed length, given by precomposition.
Given two profiles $\uc = c_1, \dots, c_n$ and $\ud = d_1, \dots, d_m$ in $C$, define (for $1\leq i \leq n$)
\[
    (c_1, \ldots, c_n) \comp{i}{} (d_1, \dots, d_m) = c_1, \dots, c_{i-1}, d_1, \dots, d_m, c_{i+1}, \dots, c_n.
\]

\begin{definition} An \mydef{operad} $P$ consists of the following data:
\begin{enumerate}[label=(O\arabic*), ref=O\arabic*]
\item a set of colors $\colors(P)$,
\item a collection of sets $P(c_1, \dots, c_n; c)$ (where $c_1, \dots, c_n, c\in \colors(P))$,
\item for each $\sigma \in \Sigma_n$, an operator \[ \sigma^* \colon P(c_1, \dots, c_n; c) = P(\uc; c) \to P(\uc \sigma; c) = P(c_{\sigma(1)}, \dots, c_{\sigma(n)}; c), \] assembling to a right action of $\Sigma_n$ on $\coprod_{\uc \in \colors(P)^n} P(\uc; c)$,
\item identity elements $\id_c \in P(c; c)$, and
\item composition operations, defined only when $c_i = d$, \[
        \comp{i}{} : P(\uc; c) \times P(\ud; d) \to P(\uc \comp{i}{} \ud; c).
    \]
\end{enumerate}
This data must satisfy associativity, unitality, and equivariance axioms \cite[\S2.2]{DrummondColeHackney:DKHTCO}.
A \mydef{map of operads} $f\colon P \to Q$ consists of a set map $\colors(P) \to \colors(Q)$ and functions 
\[ P(c_1, \dots, c_n; c) = P(\uc; c) \to Q(f\uc; fc) = Q(fc_1, \dots, fc_n; fc) \]
which are compatible with all of the basic data. 
We write $\operad$ for the category of operads.
\end{definition}

Notice that our operads are by default \emph{colored} operads.
The usual case \cite{May:GILS,Markl:OP} is recovered by considering operads with $\colors(P)$ to be a singleton set.
Likewise, our cyclic operads will also be colored.
Here we emphasize a crucial point -- the color set will be an \emph{involutive} set, that is it is equipped with an involution $c\mapsto c^\dagger$ (which is potentially trivial). 
Given two profiles $\uc = c_0, \ldots, c_n$ and $\ud = d_0, \dots, d_m$ and two integers $0\leq i \leq n$ and $0\leq j \leq m$, we define 
\[
(c_0, \ldots, c_n) \comp{i}{j} (d_0, \dots, d_m) = c_0, \dots, c_{i-1}, d_{j+1}, \dots, d_m, d_0, \dots, d_{j-1}, c_{i+1}, \dots, c_n.
\]

Operadic composition plugs an output of one operation into an input of another operation to produce a new operation. 
For cyclic operads, we do not distinguish between inputs and outputs of operations, so instead must specify which pair of `puts' to compose along.

\begin{definition}\label{def cyclic operad}
A \mydef{cyclic operad} $P$ consists the following data:
\begin{enumerate}[label=(C\arabic*), ref=C\arabic*]
\item an involutive set of colors $\colors(P)$,
\item a collection of sets $P(c_0, c_1, \dots, c_n)$, one for each profile in $\colors(P)$,
\item for each $\sigma \in \Sigma_n^+$, an operator \[ \sigma^* : P(\uc) = P(c_0, c_1, \dots, c_n) \to P(c_{\sigma(0)}, \dots, c_{\sigma(n)}) = P(\uc\sigma)\] assembling to a right action of $\Sigma_n^+$ on $\coprod_{\uc \in \colors(P)^{n+1}} P(\uc)$,
\item identity elements $\id_c \in P(c^\dagger, c)$, and
\item \label{def cyc comp}
composition operations:  
if $\uc = c_0, \dots, c_n$, $\ud = d_0, \dots, d_m$, $0\leq i \leq n$, $0\leq j \leq m$, and $c_i = d_j^\dagger$, a function
        \[
            \comp{i}{j} : P(\uc) \times P(\ud) \to P(\uc \comp{i}{j} \ud). 
        \]
\end{enumerate}
This data must satisfy commutativity, associativity, unitality, and equivariance axioms \cite[\S2.3]{DrummondColeHackney:DKHTCO}.
We additionally require that\footnote{Notice that \eqref{def co positivity} is not part of \cite[Definition 2.3]{DrummondColeHackney:DKHTCO}; this condition is called positivity there \cite[Definition 4.1]{DrummondColeHackney:DKHTCO}.
The category we call $\cycop$ is denoted by $\cycop^\uparrow$ in that work.
We could also avoid having a value of $P$ at the empty profile, so long as we also forget about the unique composition operation $\comp{0}{0}$ in \eqref{def cyc comp} when $m=n=0$.}
\begin{enumerate}[resume*]
\item the value of $P$ at the empty profile is a terminal object,  $P(\,\,) = \ast$.\label{def co positivity}
\end{enumerate}
A \mydef{map of cyclic operads} $f\colon P \to Q$ consists of an involutive function $\colors(P) \to \colors(Q)$ and functions $P(\uc) \to Q(f\uc)$
which are compatible with all of the basic data. 
We write $\cycop$ for the category of cyclic operads.
\end{definition}

In constructions we will not typically mention the empty profile $P(\,\,)$.

We can also define operads or cyclic operads enriched in some other symmetric monoidal category. 
We will only use one such alternate enriching category in this paper, namely simplicial sets.
We denote by $\soperad$ the category of \mydef{simplicial operads}. 
Objects in this category are operads enriched in simplicial sets, that is, if $P \in \soperad$ then each $P(\uc; c)$ is a simplicial set instead of a just a set, and all relevant maps in the definitions are simplicial set maps. 
We still have that $\colors(P)$ is just a set, not a simplicial set.
Likewise, $\scycop$ is the category of \mydef{simplicial cyclic operads}.

\begin{example}[Anti-involutive categories {\cite[Example 2.10]{DrummondColeHackney:DKHTCO}}]
Every small category $C$ equipped with an anti-involution $\iota \colon C^\op \to C$ satisfying $\iota \circ \iota^\op = \id_C$ gives rise to a cyclic operad $P$.
The color set of $P$ is the set of objects of $C$ (which becomes an involutive set via $\iota$), and $P(\uc)$ is empty unless the profile $\uc$ has length 0 or 2.
In the latter case, $P(c_0,c_1)$ is defined to be the set of morphisms $c_1 \to \iota c_0$.
\end{example}

There is a forgetful functor $F \colon \cycop \to \operad$ which on color sets forgets the involution and has $FP(c_1, \dots, c_n; c) \coloneqq P(c^\dagger, c_1, \dots, c_n)$.
This functor turns out to admit both a left adjoint $L$ and a right adjoint $R$, and the adjoint string $L \dashv F \dashv R$ lives over the adjoint string between sets and involutive sets from \cref{ex color sets}.
\begin{equation}\label{cd iset set adjunctions}
\begin{tikzcd}[column sep=large]
    \set^{C_2} \rar["\iota^*" description]  & \set^{\{e\}} 
    \lar[bend right=20, "\iota_!" swap] 
    \lar[bend left=20, "\iota_*"]  \end{tikzcd}
\end{equation}
If $X$ is a set, the underlying set of $\iota_!X$ will occasionally be written as 
\[
    \{ x^a \,|\, x \in X, a \in \{0,1\} \}
\]
with $(x^0)^\dagger = x^1$ and the underlying set of $\iota_*X$ as 
$\{ x_0 \times x_1 \, | \, x_i \in X \}$
with $(x_0 \times x_1)^\dagger = x_1 \times x_0$.
See \cite[\S3]{DrummondColeHackney:DKHTCO} for details on the following.\footnote{More precisely, our $LP$ coincides with $GLP$ there -- see \cite[Lemma 4.2]{DrummondColeHackney:DKHTCO}.}

\begin{definition}\label{def cycop adjoint string}
Let $F \colon \cycop \to \operad$ be the forgetful functor from cyclic operads to operads with $FP(c_1, \dots, c_n; c) \coloneqq P(c^\dagger, c_1, \dots, c_n)$.
We write $L,R \colon \operad \to \cycop$ for its left and right adjoint, which have the following partial descriptions.
\begin{itemize}[left=0pt]
\item 
Set $\colors(LP) = \iota_!\colors(P)$ and $LP(\,\,) = \ast$. 
If $c_0^{a_0}, c_1^{a_1}, \dots, c_n^{a_n} = \uc^{\underline a}$ is a profile such that there exists a \emph{unique} index $k$ with $a_k = 1$, then define 
\[
    LP(\uc^{\underline a}) \coloneqq P(c_{k+1}, \dots, c_n, c_0, \dots, c_{k-1} ; c_k).
\]
In all other cases, $LP(\uc^{\underline a})$ is defined to be the empty set.
\item
Set $\colors(RP) \coloneqq  \iota_*\colors(P)$, and let 
\[
    RP\Big(c_0^0 \times c_0^1 , \dots, c_n^0 \times c_n^1 \Big) \coloneqq \prod_{k=0}^n P(c_{k+1}^0, \dots, c_n^0, c_0^0, \dots, c_{k-1}^0; c_k^1).
\]
\end{itemize}
\end{definition}

The composition operations $\comp{i}{j}$ in $LP$ and $RP$ are induced from the composition operations $\comp{\ell}{}$ in $P$.
If $X$ is a category, then $LX$ is isomorphic to the category $X \amalg X^\op$ and $RX$ is isomorphic to $X\times X^\op$ \cite[2.7]{DrummondColeHackney:CERIMS}.

\Cref{def cycop adjoint string} applies equally well to the forgetful functor from simplicial cyclic operads to simplicial operads, $L$ and $R$ induce adjoints to $F \colon \scyc \to \soperad$; we return to this in \cref{sec QE simplicial cyclic}.

\section{Rooted and unrooted trees}\label{sec tree cats}
In this section we establish background we will need concerning tree categories.
We will use the following definition of graph with loose ends from \cite[\S3]{JoyalKock:FGNTCSM}, which is based on arcs, which are edges equipped with on of the two possible orientations.
\begin{definition}
An \mydef{undirected graph} $G$ is a diagram of finite sets
\[\begin{tikzcd}[column sep=small] A \arrow[loop left,"\dagger"] & \lar[hook'] D \rar{t} & V \end{tikzcd}\]
with $\dagger$ a fixpoint-free involution and $D \to A$ a monomorphism.
The set $A$ is the set of \mydef{arcs} and $V$ is the set of \mydef{vertices}. 
If $v\in V$ is a vertex, then $\nbhd(v) \coloneqq t^{-1}(v) \subseteq D$ is the \mydef{neighborhood} of $v$.
The \mydef{boundary} of $G$ is the set $\perim(G) = A \backslash D$, and the set of \mydef{edges}, $E$, is the set of $\dagger$-orbits $\{ [a,a^\dagger] \mid a \in A\}$.
\end{definition}

\begin{figure}
\begin{tikzpicture}
    \node[inner sep=0pt] (bg) at (0,0) {\includegraphics[width=288pt]{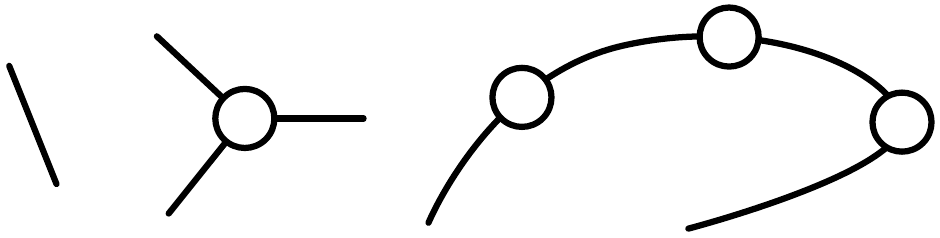}};

    \node at (-4.77, 0.7) {$a$};
    \node at (-4.2, -0.6) {$a^\dagger$};

    \node at (-3, 0.2) {$a$};
    \node at (-3.5, 0.6) {$a^\dagger$};
    \node at (-2.6, -0.55) {$b$};
    \node at (-3.5, -1) {$b^\dagger$};
    \node at (-1.1, -0.15) {$c^\dagger$};
    \node at (-1.9, 0.17) {$c$};
    \node at (-2.42, -0.02) {$v$};

    \node at (0, 0) {$0$};
    \node at (1, 0.8) {$1^\dagger$};
    \node at (-0.1, -1) {$0^\dagger$};
    \node at (2.23, 1.1) {$1$};
    \node at (4.6, 0.5) {$2$};
    \node at (3.45, 1.07) {$2^\dagger$};
    \node at (2.2, -1) {$3$};
    \node at (4.5, -0.6) {$3^\dagger$};

    \node at (0.57, 0.2) {$v_1$};
    \node at (2.8, 0.85) {$v_2$};
    \node at (4.67, -0.08) {$v_3$};

\end{tikzpicture}
\caption{An edge, a star, and a linear graph}\label{fig examples}
\end{figure}

\begin{example}\label{ex trees} We give several fundamental examples of graphs (\cref{fig examples}).
\begin{enumerate}[left=0pt]
\item If $A$ is a finite set equipped with a fixpoint-free involution, then there is an associated graph $G$ as follows.
\[\begin{tikzcd}[column sep=small] A \arrow[loop left,"\dagger"] & \lar[hook'] \varnothing \rar{t} & \varnothing \end{tikzcd}\]
This graph has no vertices, and boundary $\perim(G) = A$.
When $A$ has cardinality two, this graph is referred to as an \mydef{edge}.\label{itemex edge graph}
\item Suppose $S$ is a finite set, and let $S \amalg S^\dagger$ be the free involutive set on $S$ (\cref{ex color sets}).
There is an associated star-shaped graph $G$
\[\begin{tikzcd}[column sep=small] S \amalg S^\dagger \arrow[loop left, distance=1.3em, out=-170, in=170, "\dagger"] & \lar[hook'] S \rar{t} & \{v\} \end{tikzcd}\]
which has a single vertex $v$ with $\nbhd(v) = S$, and $\perim(G) = S^\dagger$.
This graph is referred to as a \mydef{star}.\label{itemex star graph}
\item For $n\geq 0$, let $S = \{0, 1, \dots, n\}$ and $A = S\amalg S^\dagger$.
Define $V=\{v_1, \dots, v_n\}$ and $D = A \backslash \{0^\dagger, n\}$, and let $t\colon D \to V$ be given by $t(k^\dagger) = v_k$ and $t(k) = v_{k+1}$.
This data defines the \mydef{linear graph} $\linear{n}$.
It has the property that each vertex neighborhood has two elements $\nbhd(v_k) = \{ k-1, k^\dagger \}$, as does the boundary $\perim(\linear{n}) = \{0^\dagger, n \}$. 
In the case $n=0$, the linear graph $\linear{0}$ is an edge.\label{itemex linear graph}
\end{enumerate}
\end{example}

\begin{definition}
Suppose $G$ is a graph, and let $E' \subseteq E$ and $V' \subseteq V$ be subsets of edges and vertices of $G$.
Let $A' \subseteq A$ denote the set of all arcs appearing in the edges in $E'$, and let $D' \coloneqq t^{-1}(V') \subseteq D$.
We say that $(E',V')$ is a \mydef{subgraph} if $D' \subseteq D \to A$ lands in the set $A'$.
\begin{equation*}\label{eq undirected subgraph} \begin{tikzcd}
A'\dar[hook]  \arrow[loop left,"\dagger"] & \lar[hook', dashed] D'\dar[hook]  \rar & V'\dar[hook]  \\
A \arrow[loop left,"\dagger"] & \lar[hook'] D \rar & V
\end{tikzcd} \end{equation*}
\end{definition}

\begin{definition}
A \mydef{tree} is a graph which is connected and does not contain any cycles.
\end{definition}

The edges, stars, and linear graphs from \cref{ex trees} are all trees.

\begin{convention}
Henceforth, we tacitly assume that all trees have inhabited boundary $\perim(T) \neq \varnothing$.
\end{convention}

In particular, we are excluding the star on the empty set, so $\nbhd(v)$ will always be inhabited.

\begin{definition}
If $T$ is a tree, write $\subtree(T)$ for the set of \mydef{subtrees} of $T$, that is, subgraphs of $T$ which are trees.
\end{definition}

\begin{example}\label{ex basic subtrees}
There are inclusions $V \hookrightarrow \subtree(T)$ and $E \hookrightarrow \subtree(T)$. 
A vertex $v$ maps to the star $\medstar_v$ which has a single vertex and arc set $\nbhd(v) \amalg \nbhd(v)^\dagger$.
An edge $e = [a,a^\dagger]$ maps to $\linear{0}_e$ which is the subgraph which has no vertices and exactly two arcs $a,a^\dagger$.
\end{example}

If $T$ is a tree, then there is an associated cyclic operad $\cdcatrm(T)$ which has involutive color set $\colors(\cdcatrm(T)) = A$ and is freely generated by the vertices of $T$.
The value of $\cdcatrm(T)$ at a profile $\underline{a} = a_0, \dots, a_n$ (with $n\geq 0$) is either empty or a point. 
It is inhabited just when there is a subtree $S\in \subtree(T)$ such that $\perim(S) = \{ a_0, \dots, a_n \}$ and the elements $a_0, \dots, a_n$ are distinct. 
We can think in this case $\cdcatrm(T)(\ua) = \{S\}$.
One can visualize the composition operation $\comp{i}{j}$ as a grafting of subtrees of $T$.

\begin{definition}
Suppose $T$ and $T'$ are trees. 
A \mydef{tree map} from $T$ to $T'$ is a map of cyclic operads $\cdcatrm(T) \to \cdcatrm(T')$.
The category $\cdcat$ has objects the trees and morphisms the tree maps, and is called the \mydef{cyclic dendroidal category}.
\end{definition}

There is thus a fully faithful inclusion $\cdcat \hookrightarrow \cycop$. 
The category $\cdcat$ is called $\mathsf{U}_\text{cyc}$ in other sources; it was originally defined in \cite{HackneyRobertsonYau:GCHMO} as a full subcategory of a bigger category $\mathsf{U}$ whose objects are arbitrary connected graphs.
The preceding definition is equivalent to the original one by \cite[Proposition 3.2]{Hackney:SCGO}.
There is also a different combinatorial description of tree maps in  \cite[Proposition B22]{Hackney:CGOS} or \cite[Appendix A.1]{Hackney:SCGO}.

\begin{definition}\label{def rooted tree}
A \mydef{rooted tree} is a pair $(T,r)$ where $T$ is a tree and $r\in \perim(T)$ is a chosen boundary element.
\end{definition}

We write $\out(T,r) = [r,r^\dagger] \in E$ for the output edge of a rooted tree $(T,r)$.
We also have the set of input edges $\inp(T,r) \subseteq E$, which is the image of $\perim(T) \setminus r$ under $A \twoheadrightarrow E$.
If $(T,r)$ is a rooted tree and $S$ is a subtree, we will always regard $S$ as a rooted tree in a canonical way: its root is the element of $\perim(S) \subseteq \perim(T)$ which is nearest to $r$.
In particular, each vertex $v$ in a rooted tree has an output edge $\out(v) \in E$ and a set of input edges $\inp(v) \subseteq E$.
Rooted trees admit simpler equivalent descriptions involving only vertices and edges, rather than arcs (see \cite{Weiss:FODS}, \cite{Kock:PFT}, etc.), but the above facilitates comparison with unrooted trees. 

If $(T,r)$ is a rooted tree, then there is an associated operad $\Omega(T,r)$ which has the edges of $T$ as its colors and is freely generated by the vertices.
More specifically, $\Omega(T,r)(e_1, \dots, e_n; e)$ is inhabited just when there is a subtree $S$ whose output edge is $e$, whose set of inputs is $\{e_1, \dots, e_n\} \subseteq E$ and all $e_i$ are distinct.

\begin{definition}\label{def rooted tree map}
Suppose $(T,r)$ and $(T',r')$ are rooted trees. 
A \mydef{rooted tree map} from $(T,r)$ to $(T',r')$ is a map of operads $\Omega(T,r) \to \Omega(T',r')$.
The category of rooted trees is denoted by $\dendcat$ and is called the \mydef{dendroidal category} \cite{MoerdijkWeiss:DS}.
\end{definition}

There is a fully faithful functor $\simpcat \to \dendcat$ sending $[n] = \{0, 1, \dots, n \}$ to the rooted linear tree $(\linear{n},n)$.
We will simply regard $\simpcat$ as a full subcategory of $\dendcat$, ignoring the distinction between the objects $[n]$ and $(\linear{n},n)$.
See Figure~\ref{fig delta v omega}.
\begin{figure}
\[ \begin{tikzcd}[sep=tiny]
{[3]:} & & 0 \ar[rr] & & 1 \ar[rr] & & 2 \ar[rr] & & 3 \\
{\linear{3}:} & |[circle]| \phantom{v_1} \ar[rr, no head, "0^\dagger" near start, "0" near end, very thick] & & |[draw,circle, very thick]| v_1 \ar[rr, no head, "1^\dagger" near start, "1" near end, very thick] & & |[draw,circle, very thick]| v_2 \ar[rr, no head, "2^\dagger" near start, "2" near end, very thick] && |[draw,circle, very thick]| v_3 \ar[rr, no head, "3^\dagger" near start, "3" near end, very thick] && |[circle]| \phantom{v_3}
\end{tikzcd} \]
\caption{Objects in $\simpcat$ versus their counterparts in $\dendcat$}\label{fig delta v omega}
\end{figure}

\begin{definition}\label{forgetful-functor-def}
\mydef{Root-elision} is the functor $f \colon \dendcat \to \cdcat$ which sends a rooted tree $(T,t)$ to the underlying tree $T$, and sends a rooted tree map $(S,s) \to (T,t)$ to the underlying tree map $S \to T$.
\end{definition}

If $(T,r)$ is a rooted tree, then $L\Omega(T,r)$ is isomorphic to $\cdcatrm(T)$, where $L \colon \operad \to \cyc$ is the left adjoint from \cref{def cycop adjoint string}.
Thus root-elision is more or less given by $L$.
The following appears in \cite[Lemma 6.15]{Hackney:CGOS}.
\begin{proposition}\label{disc fib}
Root-elision is a discrete fibration. \qed
\end{proposition}
A discrete fibration $p\colon D \to C$ is a Grothendieck fibration whose fibers are discrete, and when $C,D$ are small categories this just means the following square is a pullback
\[ \begin{tikzcd}
\mor(D) \rar{\cod} \dar \drar[phantom, "\lrcorner" very near start] & \ob(D)\dar  \\
\mor(C) \rar[swap]{\cod} & \ob(C)
\end{tikzcd} \]
(and dually for discrete opfibration).
\Cref{disc fib} means that any map $T \to f(T',r')$ of $\cdcat$ has a unique lift $(T,r) \to (T',r')$ (with specified codomain) in $\dendcat$. 

\begin{remark}\label{rmk xi}
The category $\cdcat$ is different from the category of trees $\mathsf{\Xi}$ from \cite{HackneyRobertsonYau:HCO} (see also \cite{Strumila:QMO}), but is better behaved.
The key difference between the two is that in $\mathsf{\Xi}$ the automorphism group of the edge is trivial, rather than $C_2$.
The functor $\dendcat \to \mathsf{\Xi}$ is not a discrete fibration and the functor $\mathsf{\Xi} \to \cyc$ is faithful but not full.
\end{remark}

\subsection{Factorization of (rooted) tree maps}

We now define certain classes of maps between (rooted) trees which will play a key role in the theory of (cyclic) dendroidal sets.

\begin{definition}\label{active-inert-def}
    A map $S \to T$ in $\cdcat$ is
    \begin{itemize}[left=0pt]
        \item \mydef{active} if the map on color sets restricts to an isomorphism $\perim (S) \cong \perim (T)$, and
        \item \mydef{inert} if it sends each generator of $\cdcatrm(S)$ (i.e.\ a vertex) to a generator of $\cdcatrm(T)$, up to the symmetric group action. 
    \end{itemize}
    A map $(S,s_0) \to (T,t_0)$ in $\dendcat$ is active or inert just when its image in $\cdcat$ is so.
    We write $\dendcat_\actrm$, $\cdcat_\actrm$ for the subcategories of active maps, and $\dendcat_\intrm$, $\cdcat_\intrm$ for the subcategories of inert maps.
\end{definition}

For example, given any subtree inclusion $S \subseteq T$, there is a canonical inert map $S \hookrightarrow T$.
Furthermore, all maps out of the edge $\linear{0}$ are inert. 
Isomorphisms of trees are both active and inert.
Note that in \cite{HackneyRobertsonYau:GCHMO}, inert maps were referred to as \emph{embeddings}; further equivalent definitions, in the context of the larger category $\mathsf{U}$, can be found in that reference and \cite{Hackney:CGOS}.

\begin{proposition}[{\cite[1.3.13]{Kock:PFT}, \cite[Prop.~5.2]{HackneyRobertsonYau:GCHMO}}]\label{active-inert-factorization}
The active and inert maps constitute an orthogonal factorization system on $\dendcat$ (respectively on $\cdcat$).
\end{proposition}
In other words, in both $\dendcat$ and $\cdcat$, every morphism factors as an active map followed by an inert map, uniquely up to unique isomorphism.

In a similar manner to \cref{disc fib}, we have the following:
\begin{proposition}\label{disc opfib}
Root-elision restricts to a discrete opfibration $f\colon \dendcat_\actrm \to \cdcat_\actrm$.
\end{proposition}
\begin{proof}
Suppose $\varphi\colon f(T,r) \to T'$ is an active map in $\cdcat$, and let $r' = \varphi(r) \in \perim(T')$.
By \cref{disc fib} there is a lift $(T,t) \to (T',r')$, but since $\varphi$ is a bijection on boundaries and $\varphi(t) = r = \varphi(r')$, we have $r' = t$, so this is a lift $(T,r) \to (T',r')$.
This lift is unique since active maps preserve roots.
\end{proof}

The category $\dendcat$ is the prototypical example of a \emph{dualizable generalized Reedy category} \cite{BergerMoerdijk:OENRC}; for brevity we say that $\dendcat$ is \emph{Reedy}.
Similarly, $\cdcat$ is Reedy \cite[Proposition 5.5]{HackneyRobertsonYau:GCHMO}.
A Reedy structure consists of an orthogonal factorization system (in our case, different from that of \cref{active-inert-factorization}) and the specification of an integral degree for each object (such that this data is suitably compatible).
A category can potentially be Reedy in multiple ways, and we now make this structure we consider more explicit.

\begin{definition}[Reedy structures]\label{def reedy structure}
Define the degree of a tree $T$ to be the cardinality of its set of vertices.
Let $\varphi \colon T \to T'$ be a tree map.
\begin{itemize}[left=0pt]
\item We say that $\varphi$ is in $\cdcat^+$ if $A \to A'$ is injective.
\item We say that $\varphi$ is in $\cdcat^-$ if $A \to A'$ is surjective and $\varphi$ is active.
\end{itemize}
Similarly, $\dendcat$ has wide subcategories $\dendcat^+ \coloneqq f^{-1}(\cdcat^+)$ and $\dendcat^- \coloneqq f^{-1}(\cdcat^-)$.
\end{definition}

In \cite{BergerMoerdijk:OENRC}, the Reedy structure on $\dendcat$ was defined in terms of injectivity/surjectivity on edges rather than arcs, but since the involutions are free this is no different.

\begin{proposition}\label{prop gen reedy}
With the above structure, the categories $\cdcat$ and $\dendcat$ are Reedy.
\end{proposition}
\begin{proof}
The statement for $\dendcat$ goes back to \cite[Example 2.8]{BergerMoerdijk:OENRC}.
The Reedy structure for $\cdcat$ appearing in \cite[Prop.\ 5.5]{HackneyRobertsonYau:GCHMO} utilizes a different degree function, given by the sum of the number of vertices and the number of internal edges, but has the same positive and negative morphisms.
The only axiom in \cite[Definition 1.1]{BergerMoerdijk:OENRC} involving the degree function is the first, so we need only verify that 
\begin{enumerate}[label=(\roman*), ref=\roman*]
\item 
non-invertible morphisms in $\cdcat^+$ (resp.\ $\cdcat^-$) raise (resp.\ lower) degree; isomorphisms in $\cdcat$ preserve the degree.
\label{Reedy def 1}
\end{enumerate}
But this follows from the same axiom for $\dendcat$.
Indeed, a map $\varphi$ in $\dendcat$ is in $\dendcat^+$ or $\dendcat^-$ if and only if $f(\varphi)$ is in $\cdcat^+$ or $\cdcat^-$.
Further, since $f$ is a discrete fibration, a map $\varphi$ is an isomorphism if and only if $f(\varphi)$ is so.
Since $f$ is a discrete fibration and surjective on objects, every map in $\cdcat$ admits some lift to $\dendcat$, and we conclude that \eqref{Reedy def 1} holds.
\end{proof}

As is usual, the opposite categories $\cdcat^\op$ and $\dendcat^\op$ are also Reedy, using the same degree function as well as the opposite factorization system $(\cdcat^\op)^+ = (\cdcat^-)^\op$ and $(\cdcat^\op)^- = (\cdcat^+)^\op$.
This allows us to consider the Reedy model structure on simplicial presheaves in \cref{sec rezk ms}.

\begin{exercise}\label{exercise EZ}
The category $\dendcat$ satisfies the following properties (see for instance \cite[6.8]{BergerMoerdijk:OENRC} and \cite[5.2]{MoerdijkNuiten:MFDS}): 
\begin{enumerate}
\item The positive maps are precisely the monomorphisms.
\item The negative maps are precisely the split epimorphisms. 
\item If two negative maps $\varphi, \varphi'$ have the same set of sections, then $\varphi = \varphi'$. 
\item Each pair of negative maps with a common domain has an absolute pushout.
\end{enumerate}
Use \cref{disc fib}, \cref{disc opfib}, and the fact that $f$ creates positive and negative maps to show that these four properties also hold for $\cdcat$.
In particular, both $\dendcat$ and $\cdcat$ are \emph{cat\'egories squelettiques} \cite[8.1.1]{Cisinski:PMTH} and \emph{Berger--Moerdijk EZ categories} \cite[Definition 6.7]{BergerMoerdijk:OENRC}.
\end{exercise}

With the Reedy structure at hand, every map $\varphi \colon T \to T'$ has a codimension given by $|V'| - |V|$. 
It is most meaningful when $\varphi$ is a positive or negative map.
\begin{definition}
A \mydef{coface map} is a positive map of codimension $1$. 
An \mydef{inner coface map} is an active coface map. 
An \mydef{outer coface map} is an inert coface map. 
A \mydef{codegeneracy map} is a negative map of codimension $-1$.      
\end{definition} 

Up to isomorphism, an inner coface $T \to T'$ is obtained by contracting an inner edge of $T'$, a codegeneracy $T \to T'$ is obtained by adding a new vertex onto the middle of some edge of $T'$, and an outer coface $T \to T'$ is obtained by grafting a star onto some boundary edge of $T$.

\begin{lemma}\label{active-outer-face}
    A map in $\cdcat$ or $\dendcat$ is active if and only if it does not factor through any outer coface on the left.
\end{lemma}

\begin{proof}
    In $\dendcat$, this is a consequence of \cite[Props.~3.9 \& 3.10]{HeutsMoerdijk:SDHT}. The statement for $\cdcat$ then follows by \cref{disc fib}.
\end{proof}

\begin{definition}\label{def boundary horn}
The \mydef{boundary} of a representable presheaf $\repu{T} = \hom(-,T) \in \cdset$, denoted $\partial \repu{T}$, is the subobject of $\repu{T}$ consisting of all maps $T' \to T$ factoring through a coface map. 
In particular, we have $\partial \linear{0} = \varnothing$.
For a coface map $\delta \colon S \to T$, the \mydef{$\delta$-horn} $\Lambda^T_\delta$ is the subobject of $\repu{T}$ consisting of all maps $T' \to T$ factoring through a coface map not isomorphic to $\delta$. 
If $\delta$ is an inner coface which contracts the inner edge $e$, will instead write  $\Lambda^T_e$ and call it an \mydef{inner horn}.
\end{definition}

Note that the definition of the boundary given above coincides with that given in \cite[\S7.1]{BergerMoerdijk:OENRC}; 
by \cite[Corollary 6.10]{BergerMoerdijk:OENRC} we have $\partial \repu{T} = \mathrm{sk}_{n-1}\repu{T}$ for $T \in \cdcat$ of degree $n$, and similarly for $(T,t) \in \dendcat$. 

\begin{remark}\label{rmk horns colimits}
Suppose $T$ is a tree and $\delta \colon S \to T$ is a face map.
Let $C_\delta \subseteq (\cdcat^+)_{/T}$ be the full subcategory on those objects $R \to T$ which are not isomorphic to $\delta$ or $\id_T$.
Then the canonical map 
$
  \colim_{(R\to T) \in C_\delta} \repu{R} \to \repu{T}
$
is isomorphic to the horn inclusion $\Lambda^T_\delta \subseteq \repu{T}$.
A similar colimit description holds for $\Lambda^{T,t}_\delta \subseteq \repo{T,t}$, as well as for boundary inclusions in either presheaf category.
Since root-elision is a fibration, $(\dendcat^+)_{/(T,t)} \to (\cdcat^+)_{/T}$ is an isomorphism of categories, and restricts to isomorphisms of the categories indexing the colimits for horns/boundaries.
Since $f_!$ preserves colimits and sends representables to representables, it sends horn inclusions to horn inclusions and boundary inclusions to boundary inclusions.
\end{remark}

We now define certain distinguished classes of monomorphisms which constitute the cofibrations in model structures on $\dset$ and $\cdset$.

\begin{definition}
    A monomorphism $X \to Y$ in $\cdset$ is \mydef{normal} if, for each $T \in \cdset$, the automorphism group of $T$ in $\cdcat$ acts freely on the non-degenerate elements of $Y_T \setminus X_T$ (in other words, if all such elements have trivial stabilizers). A cyclic dendroidal set $X \in \cdset$ is normal if the unique map $\varnothing \to X$ is a normal monomorphism (or equivalently, if for all $T \in \cdcat$, the automorphism group of $T$ acts freely on the non-degenerate elements of $X_T$). Normal monomorphisms in $\dset$ and normal dendroidal sets are defined similarly.
\end{definition}

We may note that any monomorphism having a normal codomain is normal. 
In particular, this includes the inclusion of a boundary or horn into a representable presheaf.
In fact, the boundary inclusions generate all normal monomorphisms, by \cite[8.1.35]{Cisinski:PMTH}:

\begin{proposition} \label{normal-boundary}
    In both $\cdset$ and $\dset$, the class of normal monomorphisms is the saturation of the set of boundary inclusions. \qed
\end{proposition}

\begin{proposition}
Suppose $X \to Y$ is a normal monomorphism in $\cdcat$.
Then for each $T$, the group $\aut(T)$ acts freely on $Y_T \setminus X_T$.
\end{proposition}
\begin{proof}
Similar to \cite[Proposition 1.5]{CisinskiMoerdijk:DSMHO}.
\end{proof}

\begin{definition} \label{inner-anodyne-def}
    In $\cdset$ or $\dset$, the class of \mydef{inner anodyne} morphisms is the saturation of the set of inner horn inclusions.
\end{definition}

\subsection{The Segal condition}\label{ss segal condition}
We conclude this section with the nerve theorem, which makes precise the relationship between (cyclic) operads and (cyclic) dendroidal sets.

\begin{definition}\label{def Segal core}
The \mydef{Segal core} of a tree $T \in \cdcat$ with at least one vertex is 
\[
\segcore{T} = \bigcup_{v\in T} \repu{\medstar_v} \subseteq \repu{T}
\]
running over all vertices $v$ of $T$.
Alternatively, $\segcore{T} = \colim_{U \to T} \repu{U}$ where $U$ is indexed by the subtrees coming from vertices and edges of $T$, as in \cref{ex basic subtrees}.
Similarly, we define $\segcore{T,t} \subseteq \repo{T,t}$ for $(T,t) \in \dendcat$.
\end{definition}

A cyclic dendroidal set $X \in \cdset$ is said to be \emph{Segal} if the map
\[
  X_T \cong \hom(\repu{T}, X) \to \hom(\segcore{T}, X) \cong \displaystyle \lim_{U \to T} X_U
\]
induced by the Segal core inclusion is a bijection for all $T$ having at least one vertex.
A similar definition applies for dendroidal sets.
(These are instances of more general frameworks for Segal condition: hypermoment categories \cite{Berger:MCO} and algebraic patterns \cite{ChuHaugseng:HCASC}; see \cite[\S6]{Hackney:CGOS} for details).

The inclusion $\cdcat \to \cyc$ extends to a colimit preserving functor $\cdset \to \cyc$, and its right adjoint $N \colon \cyc \to \cdset$, given by $N(P)_T = \hom(\cdcatrm(T), P)$ for $P\in \cyc$, is called the \emph{nerve functor}.
Likewise, there is a nerve functor $N \colon \operad \to \dset$.
The \emph{nerve theorem} says that $N$ is fully faithful, with essential image spanned by the Segal (cyclic) dendroidal sets \cite{Weber:F2FPRA,Elliott:Thesis}.
We consider the homotopical version of the Segal condition in \cref{sec rezk ms}.

\section{Quasi-operads and anti-involutive quasi-categories}\label{section old models}
In this section we recall two model structures which will appear in our analysis of the model structure for cyclic quasi-operads.

\begin{definition}\label{pseudo-generators-def}
A set of \mydef{pseudo-generating acyclic cofibrations} for a model category $\catM$ is a set of cofibrations $J$ such that a morphism in $\catM$ with fibrant codomain is a fibration if and only if it has the right lifting property against all maps in $J$.
\end{definition}

In particular, if $J$ is a set of pseudo-generating acyclic cofibrations for $\catM$, then the fibrant objects of $\catM$ are precisely those having the left lifting property against all maps in $J$.

\begin{example}\label{Joyal-pseudo-gen}
    A set of pseudo-generating acyclic cofibrations for the Joyal model structure on $\sset$ is given by the inner horn inclusions $\Lambda^n_i \hookrightarrow \Delta^n, 1 < i < n$ together with the endpoint inclusion $\{0\} \hookrightarrow \freeiso$, where $\freeiso$ denotes the nerve of the walking isomorphism category $\{0 \cong 1\}$.
\end{example}

Pseudo-generating acyclic cofibrations often allow for a convenient proof that an adjunction is Quillen.

\begin{lemma}\label{peusdo-gen-Quillen}
    Let $L : \catM \rightleftarrows \catN : R$ be an adjunction of model categories, and let $J$ be a set of pseudo-generating acyclic cofibrations for $\catM$. If $L$ preserves cofibrations and sends all maps of $J$ to acyclic cofibrations, then the adjunction $L \dashv R$ is Quillen.
\end{lemma}

\begin{proof}
    This is an immediate consequence of \cite[Proposition 7.15]{JoyalTierney:QCSS}.
\end{proof}

The following well-known result allows for straightforward construction of generating cofibrations and (pseudo)-generating acyclic cofibrations for right-induced model structures.

\begin{lemma}\label{right-induced-generators}
    Let $L : \catM \rightleftarrows \catN : R$ be an adjunction between model categories, such that the model structure on $\catN$ is right-induced by $R$ from that of $\catM$. If $I$ is a set of generating cofibrations (resp.~generating acyclic cofibrations, pseudo-generating acyclic cofibrations) for $\catM$, then $LI$ is a set of generating cofibrations (resp.~generating acyclic cofibrations, pseudo-generating acyclic cofibrations) of $\catN$.
\end{lemma}

\begin{proof}
    A map $f$ in $\catN$ is an acyclic fibration (resp. fibration, fibration with fibrant codomain) if and only if $Rf$ belongs to the corresponding class of $\catM$. This holds if and only if $Rf$ has the right lifting property against all maps of $I$. This, in turn, holds if and only if $f$ has the right lifting property against all maps of $LI$.
\end{proof}

\begin{definition}\label{def quasi-operad}
A dendroidal set $X \in \dset$ is a \mydef{quasi-operad} if it has the right lifting property with respect to all inner horn inclusions $\Lambda^{T,t}_e \hookrightarrow \repo{T,t}$.
\end{definition}

These were called $\infty$-operads by Cisinski and Moerdijk, following Lurie's nomenclature for quasi-categories.
They established the following model structure on dendroidal sets in \cite{CisinskiMoerdijk:DSMHO}, which they called the \emph{operadic model structure}.
See \cite[\S9.2]{HeutsMoerdijk:SDHT} for additional details and information, including a characterization of weak equivalences between cofibrant objects.

\begin{definition}[Model structure for quasi-operads]\label{omega-model-structure}

The category $\dset$ admits the \mydef{model structure for quasi-operads}, with cofibrations the normal monomorphisms and fibrant objects the quasi-operads.
The inner horn inclusions $\Lambda^{T,t}_e \hookrightarrow \repo{T,t}$ 
and endpoint inclusion $\{0\} \hookrightarrow \freeiso$ form a set of pseudo-generating acyclic cofibrations.
The boundary inclusions $\partial \repo{T,t} \to \repo{T,t}$ form a set of generating cofibrations.
\end{definition}

We end this section with a discussion of anti-involutive simplicial sets, and a Joyal-style model structure for them.
Let $\cscat \subset \cdcat$ denote the full subcategory on the linear trees $\linear{n}$ from \cref{ex trees}\eqref{itemex linear graph}. 
This category appears with a different description in \cite[\S4]{DrummondColeHackney:CERIMS} as having objects the sets $[n] = \{0, \dots, n\}$ and morphisms the disjoint union of the order preserving maps $[n] \to [m]$ and the order reversing maps $[n] \to [m]$ (so each constant function appears twice). 
It is also the category associated to the reflexive crossed simplicial group \cite[Example 2]{FiedorowiczLoday:CSGAH}. 
As with all crossed simplicial groups, $\cscat$ contains $\simpcat$ as a wide subcategory; in our setup, this inclusion is the restriction of $f\colon \dendcat \to \cdcat$.

By \cite[Proposition 4.14]{DrummondColeHackney:CERIMS}, the presheaf category $\isset$ is equivalent to the category of \emph{anti-involutive simplicial sets}, whose objects are simplicial sets $X$ equipped with anti-involutions $X^\op \xrightarrow{\cong} X$ and whose morphisms are simplicial set maps respecting the chosen anti-involutions. 
For a presheaf $X \colon \cscat^\op \to \set$, the corresponding simplicial set is given by the restriction of $X$ to $\simpcat \subseteq \cscat$, with the anti-involution acting on $X_n$ by pre-composition with the nontrivial automorphism of $\linear{n}$ (or the order-reversing isomorphism $[n] \to [n]$ in the above description).

The left adjoint $f_!$ to the restriction functor $f^* \colon \isset \to \sset$ takes the following concrete form: a simplicial set $X$ is sent to $X \amalg X^\op$ equipped with the involution coming from the swap map $[X \amalg X^\op]^\op = X^\op \amalg X \to X \amalg X^\op$.

\begin{example}\label{example cs horn boundary}
The map $f_!$ takes a horn inclusion $\Lambda^n_i \subset \Delta^n$ to the $\cscat$-horn inclusion $\Lambda^n_i \subset \nabla^n$.
Under the description above, the underlying simplicial set map is \[ \Lambda^n_i \amalg (\Lambda^n_i)^\op  \to \Delta^n \amalg (\Delta^n)^\op, \]
which is isomorphic to $\Lambda^n_i \amalg \Lambda^n_{n-i} \to \Delta^n \amalg \Delta^n$.
A similar consideration applies to boundary inclusions.
\end{example}

\begin{definition}\label{def ai quasi-cat}
An object $X \in \isset$ is a \mydef{anti-involutive quasi-category} if it has the right lifting property with respect to all inner horn inclusions.
\end{definition}

\begin{example}\label{def double J}
The anti-involutive simplicial set $f_!(\freeiso)$ is the nerve of the anti-involutive category with four objects $0,1,0^\dagger,1^\dagger$ and unique maps $a \to b$ and $a^\dagger \to b^\dagger$ for $a,b \in \{0,1\}$. 
We call the image of the endpoint inclusion $\{0 \} \to \freeiso$ under $f_! \colon \isset \to \sset$, namely $\{ 0, 0^\dagger \} \to f_!(\freeiso)$, the \mydef{double $\freeiso$-inclusion}.
\end{example}

The following model structure appears in \cite[Corollary 4.9]{DrummondColeHackney:CERIMS}, and is right-induced from the Joyal model structure on $\sset$. 

\begin{definition}\label{nabla-model-structure}
The category $\isset$ admits the \mydef{model structure for anti-involutive quasi-categories}, with cofibrations the normal monomorphisms and fibrant objects the anti-involutive quasi-categories.
The weak equivalences are those maps whose underlying map of simplicial sets is a weak categorical equivalence.
The boundary inclusions form a set of generating cofibrations.
The inner horn inclusions and the double $\freeiso$-inclusion $\{0,0^\dagger\} \to f_!(\freeiso)$ form a set of pseudo-generating acyclic cofibrations.
\end{definition}

The last statement follows from \cref{Joyal-pseudo-gen} and \cref{right-induced-generators}.

\section{Cyclic quasi-operads} \label{sec cyclic quasi-operads}
Our goal in this section is to construct a model structure on $\cdset$ analogous to the Cisinski--Moerdijk model structure for quasi-operads (\cref{omega-model-structure}). 
Our approach is to use \cref{triple-induced-model-structures} for the following adjoint string associated to root-elision $f\colon \dendcat \to \cdcat$.
\[
\begin{tikzcd}[column sep=large]
    \cdset \rar["f^*" description] 
    \rar[phantom, bend right=18, "\scriptscriptstyle\perp"]  
    \rar[phantom, bend left=18, "\scriptscriptstyle\perp"] 
    & \dset
    \lar[bend right=30, "f_!" swap, start anchor = {[yshift=1.5ex]west}, end anchor = {[yshift=1.5ex]east}] 
    \lar[bend left=30, "f_*", start anchor = {[yshift=-1.5ex]west}, end anchor = {[yshift=-1.5ex]east}] 
\end{tikzcd}
\]
We carefully study this adjoint string, with a particular focus on the adjunction $f_! \dashv f^*$.
The following appears as \cite[Remark 6.21]{Hackney:CGOS}.

\begin{proposition}\label{f-shriek-formula}
For $X \in \dset$, the cyclic dendroidal set $f_! X$ is defined on objects by the following formula:
\[
    (f_!X)_T = \sum_{r \in \perim T} X_{T,r}
\]
Given a morphism $S \to T$ in $\cdcat$, the induced map $(f_! X)_T \to (f_! X)_S$ acts on the coproduct component corresponding to $t \in \perim T$ as the composite 
\[
X_{T,t} \to X_{S,s} \hookrightarrow \sum_{r \in \perim S} X_{S,r}
\]
where $s$ is the unique root of $S$ compatible with $t$ (i.e., $(S,s) \to (T,t)$ is the unique lift of $S \to T$ in $\dendcat$ with codomain $(T,t)$), and the right-hand map above is the coproduct inclusion corresponding to $s$. \qed
\end{proposition}

The preceding formula immediately implies that $f_!$ preserves monomorphisms, and that each leg of the unit $\eta \colon \id_{\dset} \Rightarrow f^* f_!$ is a monomorphism.
This latter fact implies in particular that $f^*f_!$ and $f_!$ are faithful functors.
Recall from \cref{rmk horns colimits} that $f_!$ preserves horn and boundary inclusions.
The following, which explains the behavior of $f^*$ on boundary inclusions and inner horn inclusions, is the key technical lemma of this section.

\begin{lemma} \label{boundary horn pushout}
The diagram below left is a pushout in $\dset$ for each tree $T \in \cdcat$, and the diagram below right is a pushout for each inner coface map $\delta$ with codomain $T$.
\[ \begin{tikzcd}
  \sum\limits_{t\in \perim T} \partial \repo{T,t} \rar[hook]  \dar \ar[dr, phantom, "\ulcorner" very near end] & 
  \sum\limits_{t\in \perim T} \repo{T,t}  \dar 
&
\sum\limits_{t\in \perim T} \Lambda^{T,t}_\delta \rar[hook]  \dar \ar[dr, phantom, "\ulcorner" very near end] & 
  \sum\limits_{t\in \perim T} \repo{T,t}  \dar \\
  f^* \partial \repu{T} \rar[hook] & f^*\repu{T}
& 
  f^* \Lambda^T_\delta \rar[hook] & f^*\repu{T}
\end{tikzcd} 
\]
In both cases, the vertical maps are sums of the unit of the $f_! \dashv f^*$ adjunction.
\end{lemma}

\begin{proof}
The bottom maps in these diagrams are monomorphisms, since $f^*$ is a right adjoint and hence preserves monomorphisms.
We focus on the inner horn diagram to the right, as the case for boundaries is a minor modification.
It suffices to show that 
\begin{equation}\label{eq complements bij}
\begin{tikzcd}[sep=small]
\sum_t \left( \repo{T,t}_{S,s} \setminus (\Lambda^{T,t}_\delta)_{S,s} \right) \rar & (f^*\repu{T})_{S,s} \setminus (f^* \Lambda^T_\delta)_{S,s}
\end{tikzcd}
\end{equation}
is a bijection for each $(S,s) \in \dendcat$.
Suppose $\varphi \colon S \to T$ is in $(f^*\repu{T})_{S,s} \setminus (f^* \Lambda^T_\delta)_{S,s}.$
Then $\varphi$ does not factor through an outer face map.
Since $\varphi$ is active by \cref{active-outer-face} and $f\colon \dendcat_\actrm \to \cdcat_\actrm$ is a discrete opfibration (\cref{disc opfib}), there is a unique $t$ with $\varphi \colon (S,s) \to (T,t)$ in $\dendcat$.
Thus \eqref{eq complements bij} is a bijection, as desired.
\end{proof}

As sums and pushouts preserve cofibrations and inner anodyne maps, we deduce:

\begin{corollary}\label{f-star-boundary-horn}
The functor $f^* \colon \cdset \to \dset$ sends boundary inclusions to cofibrations, and inner horn inclusions to inner anodyne maps. \qed
\end{corollary}

\begin{proposition}\label{f-shriek-star-composite-Quillen}
    The adjunction $f^* f_! : \dset \rightleftarrows \dset : f^* f_*$ is Quillen.
\end{proposition}

\begin{proof}
This is an application of \cref{peusdo-gen-Quillen}, using \cref{rmk horns colimits}, \cref{normal-boundary}, \cref{boundary horn pushout}, \cref{f-star-boundary-horn}, and the fact that $f^*f_!(\{0\}\to \freeiso)$ is isomorphic to a coproduct of two copies of $\{0\}\to \freeiso$, hence is an acyclic cofibration.
\end{proof}

\begin{theorem}\label{cdset-right-induced}
    The category of cyclic dendroidal sets $\cdset$ admits a model structure $\cdset_r$ right-induced along $f^* \colon \cdset \to \dset$ from the model structure for quasi-operads. This model structure may be characterized as follows:
    \begin{itemize}[left=0pt]
        \item Cofibrations are normal monomorphisms;
        \item Fibrant objects are cyclic dendroidal sets having the right lifting property against all inner horn inclusions $\Lambda^T_e \hookrightarrow \repu{T}$.
        \item A map with fibrant codomain is a fibration if and only if it has the right lifting property with respect to all inner horn inclusions $\Lambda^T_e \hookrightarrow \repu{T}$ and the double $\freeiso$-inclusion $\{0,0^\dagger\} \to f_!(\freeiso)$ from \cref{def double J}.
    \end{itemize}
\end{theorem}

\begin{proof}
    The existence of the model structure follows from  \cref{triple-induced-model-structures} and \cref{f-shriek-star-composite-Quillen}, while the characterization of cofibrations and a set of pseudo-generating acyclic cofibrations follows from \cref{right-induced-generators} and the results listed in the proof of \cref{f-shriek-star-composite-Quillen}.
    For the characterization of fibrant objects, an object $X \in \cdset$ is fibrant if and only if $f^* X \in \dset$ has the right lifting property against all inner horn inclusions $\Lambda^{T,t}_e \hookrightarrow \repo{T,t}$. 
    This, in turn, holds if and only if $X$ has the right lifting property against all maps $f_! \Lambda^{T,t}_e \hookrightarrow f_! \repo{T,t}$, and these are precisely the inner horn inclusions in $\cdset$.
\end{proof}

\begin{theorem}\label{cdset-left-induced}
    The category of cyclic dendroidal sets $\cdset$ admits a model structure $\cdset_l$ left-induced along $f^* \colon \cdset \to \dset$ from the model structure for quasi-operads (\cref{omega-model-structure}).
    The cofibrations are the monomorphisms $X \to Y$ such that every non-identity automorphism of a tree $T$ which fixes at least one boundary element of $T$ does not fix any nondegenerate element of $Y_T \setminus X_T$. 
\end{theorem}

\begin{proof}
    The existence of the model structure follows from \cref{triple-induced-model-structures} and  \cref{f-shriek-star-composite-Quillen}. 
    The characterization of the cofibrations follows from the characterization of the cofibrations of $\dset$ as normal monomorphisms, together with the fact that the automorphisms of a rooted tree $(T,t)$ consist precisely of the automorphisms of $T$ which fix the root $t$.
\end{proof}

Notice that the identity map constitutes a Quillen equivalence $\cdset_r \rightleftarrows \cdset_l$ by \cite[Proposition 1.16]{HackneyRovelli:IMSHC}.
The following example shows that the model structures $\cdset_r$ and $\cdset_l$ are distinct, with $\cdset_l$ having a larger class of cofibrations.

\begin{example}
Let $X\in \isset$ be any anti-involutive simplicial set which is not normal (e.g.\ $X = \Delta^0$ equipped with the unique anti-involution).
Of course $X$, when viewed as a cyclic dendroidal set, is not cofibrant in $\cdset_r$. 
However, it is cofibrant in $\cdset_l$ because $f^*X \in \dset$ is in the image of $\sset \to \dset$, hence is normal.
\end{example}

\section{Cyclic quasi-operads vs. simplicial cyclic operads}\label{sec QE simplicial cyclic}
Cisinski and Moerdijk showed that there is a homotopy coherent nerve/rigidification adjunction
\[
\rigc : \dset \rightleftarrows \soperad : \hNerve
\]
(induced from the Boardman--Vogt W-construction for operads) which is a Quillen equivalence between the model structure for quasi-operads and the model structure for simplicial operads  \cite[Theorem 8.15]{CisinskiMoerdijk:DSSO}, an enhancement of the usual equivalence between simplicial sets and simplicial operads.
Recall from \cref{sec operads and cyclic} that there is an adjoint string 
\[
\begin{tikzcd}[column sep=large]
    \scycop  \rar["F" description] 
    \rar[phantom, bend right=18, "\scriptscriptstyle\perp"]  
    \rar[phantom, bend left=18, "\scriptscriptstyle\perp"] 
    & \soperad.
    \lar[bend right=30, "L" swap, start anchor = {[yshift=1.5ex]west}, end anchor = {[yshift=1.5ex]east}] 
    \lar[bend left=30, "R", start anchor = {[yshift=-1.5ex]west}, end anchor = {[yshift=-1.5ex]east}] 
\end{tikzcd}
\]
where $\soperad$ is the category of simplicial operads and $\scycop$ is the category of simplicial cyclic operads.
Our aim is to lift the Quillen equivalence $\rigc \dashv \hcn$ along $F$. 

\begin{definition}\label{def dk and fibrations}
Suppose that $g \colon P \to Q$ is a map in $\scycop$ or $\soperad$.
We say that $g$ is a \mydef{Dwyer--Kan equivalence} (resp.\ \mydef{isofibration}) if
\begin{enumerate}
    \item $g$ is locally a weak homotopy equivalence (resp.\ Kan fibration) of simplicial sets, and\label{item local}
    \item the induced functor on underlying categories is an equivalence of categories (resp.\ isofibration of categories).\label{item underlying cat}
\end{enumerate}
\end{definition}

`Locally' in \eqref{item local} means that for each list of colors $c_0, c_1, \dots, c_n \in \colors(P)$, the induced map of simplicial sets
\[
    P(c_0, \dots, c_n) \to Q(g c_0, \dots, g c_n) \,\, \Big( \text{resp.\ } P(c_1, \dots, c_n; c_0) \to Q(g c_1, \dots, g c_n; g c_0) \Big)
\]
is a weak homotopy equivalence or Kan fibration.
The underlying category mentioned in \eqref{item underlying cat} is obtained by first forgetting down from $\scycop$ or $\soperad$ to simplicial categories, then taking path components of each hom object.
\[ \begin{tikzcd}[sep=small]
\scycop \rar{F} & \soperad \rar & \scat \rar{\pi_0} & \cat
\end{tikzcd} \]

\begin{definition}[Model structure on simplicial (cyclic) operads] \label{sop-scycop-model-str-def}
The categories $\scycop$ and $\soperad$ admit model structures whose weak equivalences are the Dwyer--Kan equivalences and fibrations are the isofibrations.
\end{definition}

The model structure on $\soperad$ is due to Cisinski and Moerdijk \cite[Theorem 1.14]{CisinskiMoerdijk:DSSO}, and is right proper and combinatorial.
The model structure on $\scycop$ is constructed in \cite[Theorem 6.3]{DrummondColeHackney:DKHTCO}, where it is right-induced from the one on simplicial operads via the functor $F$, using \cref{triple-induced-model-structures}. 
As we wish to show that $\scycop$ is Quillen equivalent to the model structures on $\cdset$ of \cref{sec cyclic quasi-operads}, we next introduct an analogue of the homotopy coherent rigidification/nerve adjunction which relates simplicial sets and simplicial categories. 
Athough various descriptions of these functors exist in the literature, our presentation will most closely follow that given by Lurie \cite[Definition 1.1.5.3]{Lurie:HTT} in the categorical case.

We begin by defining the functor $\rigc \colon \cdcat \to \scycop$. 
Given a tree $T \in \cdcat$ and a color profile $\ua = a_1, \dots, a_n$ (i.e.\ a list of arcs of $T$), we say that $\ua$ \emph{bounds a subtree $S$ of $T$} if the $a_i$ are distinct and $\perim(S) = \{ a_1, \dots, a_n \}$.
This occurs if and only if $\cdcatrm(T)(\ua)$ is nonempty.
We define $\rigc T \in \scycop$ as follows.

\begin{itemize}
    \item Colors of $\rigc T$ are given by arcs of $T$, with the inherited involution.
    \item Given a color profile $\underline{a}$, the simplicial set $\rigc T(\underline{a})$ is empty if $\ua$ does not bound a subtree of $T$.
    If $\ua$ bounds a subtree $S$ of $T$, then we may consider the power set of the set of edges $E_S$ as a poset under inclusion.
        Let $\Pbd{S}$ denote the subposet of $P(E_S)$ consisting of those subsets which contain all boundary edges. 
        Then $\rigc T(\underline{a})$ is the nerve of $\Pbd{S}$.
\end{itemize}

To describe the composition maps in $\rigc T$, consider a pair of color profiles $\underline{a} = a_0,\ldots,a_n$ and $\underline{b} = b_0,\ldots,b_m$ with $a_i = b_j^\dagger$ for some $1 \leq i \leq n$ and $1 \leq j \leq m$. 
\begin{itemize}
    \item If either $\underline{a}$ or $\underline{b}$ does not bound a subtree of $T$, then $\rigc T(\underline{a}) \times \rigc T(\underline{b}) = \varnothing$, so the composition map is the unique map $\varnothing \to \rigc T(\underline{a} \comp{i}{j} \underline{b})$.
    \item If $\unda$ and $\underline{b}$ bound subtrees $R$ and $S$, then $R$ and $S$ overlap in a single edge $e = [a_i, b_j]$. 
    The color profile $\underline{a} \comp{i}{j} \underline{b}$ bounds the subtree $R \cup S$. 
    In this case, we take the composition map $\comp{i}{j}$ to be induced by the morphism $\Pbd{R} \times \Pbd{S} \to \Pbd{R \cup S}$ which takes a pair of subsets to their union.
\end{itemize}

Geometrically, $\Pbd{S}$ is a cube whose dimension is equal to the number of internal edges of $S$. We further define the \mydef{boundary} $\partial N\Pbd{S}$ to be the subcomplex of $\rigc T(\underline{a})$ consisting of all simplices which do not include both the initial and terminal elements of $\Pbd{S}$. 

In particular, for an arc $a$ of $T$, the profile $(a^\dagger,a)$ bounds the edge $[a^\dagger,a]$. 
As all arcs of this subtree are part of its boundary, we have $\rigc T(a^\dagger,a) \cong \Delta^0$; the identity at $a$ is the unique morphism from $\Delta^0$ to this operation space.

Given a tree map $\varphi \colon T \to T'$, the induced map of simplicial operads $\rigc \varphi \colon \rigc T \to \rigc T'$ acts as $\varphi$ does on color sets. 
To define the induced map on spaces of operations, we note that if $\underline{a}$ bounds a subtree $S$ of $T$, then $\varphi \underline{a}$ bounds a subtree $S'$ of $T'$ and induces a map $\Pbd{S} \to \Pbd{S'}$. 
Then $\varphi$ acts on $\rigc T(\underline{a})$ as the induced map between nerves.

The functor $\rigc \colon \dendcat \to \soperad$ admits a similar description, with the colors of $\rigc(T,t)$ given by edges of $T$ and the operation space $\rigc (T,t)(\underline{e};e_0)$ given by the nerve of the poset of subsets of internal edges of the subtree bounded by $(\underline{e};e_0)$ (or empty if no such tree exists).
See \cite[\S2.7.7]{HeutsMoerdijk:SDHT}.

\begin{remark}
The functors denoted by $\rigc$ above may be seen as instances of the W-construction, a functor $W \colon \soperad \to \soperad$, and its cyclic analogue $W \colon \scycop \to \scycop$. Specifically, the rigidification functors are the composites $\dendcat \hookrightarrow \operad \hookrightarrow \soperad \xrightarrow{W} \soperad$ and $\cdcat \hookrightarrow \cycop \hookrightarrow \scycop \xrightarrow{W} \scycop$. 
For a detailed discussion of the W-construction, see \cite[\S1.7]{HeutsMoerdijk:SDHT} for the non-cyclic case and \cite[Ch.\ 8]{Elliott:Thesis} for the cyclic case. 
(See also \cite{Lukacs:CODSHC,YauJohnson:BVRGP}.)
\end{remark}

\begin{proposition}[{\cite[Lemma 8.2.7]{Elliott:Thesis}}]
\label{elliott 827}
There is a natural isomorphism $L \rigc \cong \rigc f$ of functors from $\dendcat$ to $\scycop$.
\end{proposition}
\begin{proof}
Given a tree $(T,t)$, we define a particular isomorphism 
\[
  \colors(L\rigc (T,t)) = \iota_! E \to A = \colors(\rigc  T) = \colors(\rigc f(T,t)).
\]
The color set of $L\rigc (T,t)$ is $\{e^u \mid e\in E, u \in \{0,1\} \}$ equipped with the involution $e^0 \overset\dagger\leftrightarrow e^1$ (using the notation above \cref{def cycop adjoint string}).
At an edge $e \in E$, the isomorphism sends $e^0$ to the arc of $e$ pointing away from the root and $e^1$ to the arc of $e$ pointing towards the root.
In particular, if $e = [t, t^\dagger]$ is the root edge, then $e^1 \mapsto t$.

Fix a profile ${\underline e}^{\underline u} = (e_0^{u_0}, \dots, e_n^{u_n})$ in $\colors(L\rigc (T,t))$ as well as its image $\underline a = (a_0, \dots, a_n)$ in $A_T$.
Suppose $L\rigc (T,t)({\underline e}^{\underline u})$ is nonempty.
Then 
\begin{enumerate}[label=(\Roman*), ref=\Roman*]
\item there is exactly one $u_k$ which is equal to $1$, and
\item since \[ L\rigc (T,t)({\underline e}^{\underline u}) = 
\rigc(T,t)(e_{k+1}, \dots, e_n, e_0, \dots, e_{k-1} ; e_k)\]
is nonempty, there is a subtree $S$ with $n$ input edges $e_0, \dots, \hat e_k, \dots, e_n$ and output edge $e_k$. 
Furthermore, this simplicial set is $N\Pbd{S}$.\label{item L description}
\end{enumerate}
In this case, $\perim(S) = \{a_0, \dots, a_n\}$ and $N\Pbd{S} = \rigc T(\ua)$.

On the other hand, if $\rigc T(\ua)$ is nonempty, then there is a subtree $S$ with $\perim(S) = \{a_0, \dots, a_n\}$ an $n+1$-element set.
Exactly one $a_k$ will point towards the root $t$, which implies that exactly one $u_k$ is equal to 1.
Thus the equation in \eqref{item L description} holds, and we again have $L\rigc (T,t)({\underline e}^{\underline u}) = N\Pbd{S} = \rigc T(\ua)$.
Moreover, the cyclic operad structure on both sides is just given by grafting of subtrees, so this establishes the isomorphism $L\rigc (T,t) \cong \rigc f(T,t)$.
\end{proof}

By left Kan extension, we obtain  functors $\rigc \colon \cdset \to \scycop$ and $\rigc \colon \dset \to \soperad$; we will refer to these as the \mydef{rigidification} functors. 
By standard arguments, each has a right adjoint $\hNerve$, which we will refer to as the \mydef{homotopy coherent nerve}. 
For $\mathcal{O} \in \scycop$ and $T \in \cdcat$, we have
\[
\hNerve(\mathcal{O})_T = \scycop(\rigc T, \mathcal{O})
\]
with structure maps induced by pre-composition, and similarly for $\hNerve \colon \soperad \to \dset$.

\begin{convention}
Below we will perform a number of 2-categorical manipulations, especially involving of mates of natural transformations \cite[\S2]{KellyStreet:RE2C}.
For concision and legibility we use string diagrams for this purpose, roughly following \cite{JoyalStreet:GTC1,HinzeMarsden:ISD,Lauda:FAAA}.
Our string diagrams flow down on the page, and we write an unadorned cap $\cap$ for the unit of an adjunction and an unadorned cup $\cup$ for the counit of an adjunction.
\end{convention}

By passing to free cocompletions, \cref{elliott 827} implies the following square (whose 1-cells are all left adjoints) commutes up to natural isomorphism.
\[
\begin{tikzcd}
    \dset \ar[r,"\rigc"] \ar[d,swap,"f_!"] & \soperad \ar[d,"L"] \ar[dl, phantom, "{\scriptstyle\Swarrow} \kappa"] \\
    \cdset \ar[r,"\rigc"] &  \scycop
\end{tikzcd}
\]
Of course $L \rigc \dashv \hcn F$ and $\rigc f_! \dashv f^* \hcn$ are adjunctions, so $\kappa \colon L \rigc \cong \rigc f_!$ induces an isomorphism $\alpha \colon f^* \hNerve \cong \hNerve F$. 
The usual construction is depicted in \cref{fig alpha} (see also \cite[Lem.~8.3.5]{Elliott:Thesis}), where 
$\alpha$ is obtained by composing $\kappa$ with the unit of the $L \rigc \dashv \hcn F$ adjunction and the counit of the $\rigc f_! \dashv f^* \hcn$ adjunction, and these units and counits are themselves composites of units and counits from the original four adjunctions.
\begin{figure}
\begin{tikzpicture}
	\begin{pgfonlayer}{nodelayer}
		\node [style=short box] (0) at (3.75, 0) {$\kappa$};
		\node [style=none] (1) at (4.25, 0.5) {};
		\node [style=none] (2) at (3.25, 0.5) {};
		\node [style=none, label={[label distance=-2]180:$\rigc$}] (5) at (4.25, 1.5) {};
		\node [style=none, label={[label distance=-2]180:$L$}] (6) at (3.25, 1.25) {};
		\node [style=none] (7) at (2, 1.25) {};
		\node [style=none] (13) at (2, -1.5) {};
		\node [style=none] (14) at (3.25, -0.5) {};
		\node [style=none] (15) at (4.25, -0.5) {};
		\node [style=none, label={[label distance=-2]0:$f_!$}] (17) at (4.25, -1.25) {};
		\node [style=none] (18) at (5.5, -1.25) {};
		\node [style=none] (30) at (0, 0) {$\coloneqq$};
		\node [style=short box] (37) at (-1.75, 0) {$\alpha$};
		\node [style=none] (38) at (-1.25, 0.5) {};
		\node [style=none] (39) at (-2.25, 0.5) {};
		\node [style=none] (40) at (-1.25, 1.5) {};
		\node [style=none] (41) at (-2.25, 1.5) {};
		\node [style=none] (44) at (-2.25, -0.5) {};
		\node [style=none] (45) at (-1.25, -0.5) {};
		\node [style=none] (46) at (-2.25, -1.5) {};
		\node [style=none] (47) at (-1.25, -1.5) {};
		\node [style=none] (54) at (5.5, 1.5) {};
		\node [style=none] (76) at (1, 1.5) {};
		\node [style=none] (77) at (1, -1.5) {};
		\node [style=none, label={[label distance=-2]0:$\rigc$}] (78) at (3.25, -1.5) {};
		\node [style=none] (79) at (6.5, -1.5) {};
		\node [style=none] (80) at (6.5, 1.5) {};
		\node [style=none] (81) at (-2.25, -2) {\strut$\hcn$};
		\node [style=none] (82) at (-1.25, -2) {\strut$F$};
		\node [style=none] (83) at (-2.25, 2) {\strut$f^*$};
		\node [style=none] (84) at (-1.25, 2) {\strut$\hcn$};
		\node [style=none] (85) at (1, -2) {\strut$\hcn$};
		\node [style=none] (86) at (2, -2) {\strut$F$};
		\node [style=none] (87) at (5.5, 2) {\strut$f^*$};
		\node [style=none] (88) at (6.5, 2) {\strut$\hcn$};
	\end{pgfonlayer}
	\begin{pgfonlayer}{edgelayer}
		\draw (5.center) to (1.center);
		\draw (6.center) to (2.center);
		\draw [bend left=90, looseness=1.75] (7.center) to (6.center);
		\draw (7.center) to (13.center);
		\draw (17.center) to (15.center);
		\draw [bend left=90, looseness=1.75] (18.center) to (17.center);
		\draw (40.center) to (38.center);
		\draw (41.center) to (39.center);
		\draw (46.center) to (44.center);
		\draw (47.center) to (45.center);
		\draw (18.center) to (54.center);
		\draw [bend left=90, looseness=1.75] (76.center) to (5.center);
		\draw (76.center) to (77.center);
		\draw [bend left=90, looseness=1.75] (79.center) to (78.center);
		\draw (79.center) to (80.center);
		\draw (78.center) to (14.center);
	\end{pgfonlayer}
\end{tikzpicture}
\caption{The isomorphism $f^* \hNerve \cong \hNerve F$}\label{fig alpha}
\end{figure}

Our approach to proving the desired Quillen equivalence will involve a mate of $\kappa$, namely the natural transformation $\beta \colon \rigc f^* \Rightarrow F \rigc$ depicted in Figure~\ref{fig beta}.
\begin{figure}
\begin{tikzcd}
\cdset \ar[r,"f^*"] \ar[dr, bend right=15,swap,"\id"] \ar[dr, phantom, bend left=15, "{\scriptstyle\Swarrow} \varepsilon"] &   \dset \ar[r,"\rigc"] \ar[d,swap,"f_!"]  & \soperad \ar[d,"L"] \ar[dl, phantom, "{\scriptstyle\Swarrow} \kappa"] \ar[dr,bend left=15,"\id"] \ar[dr,phantom, bend right = 15, "{\scriptstyle\Swarrow} \eta"] \\
 {} &  \cdset \ar[r,swap,"\rigc"] &  \scycop \ar[r,swap,"F"] & \soperad
\end{tikzcd}
\qquad 
\begin{tikzpicture}
	\begin{pgfonlayer}{nodelayer}
		\node [style=short box] (0) at (2.75, 0) {$\kappa$};
		\node [style=none] (1) at (3.25, 0.5) {};
		\node [style=none] (2) at (2.25, 0.5) {};
		\node [style=none] (5) at (3.25, 1.5) {};
		\node [style=none, label={[label distance=-2]180:$L$}] (6) at (2.25, 1.25) {};
		\node [style=none] (7) at (1, 1.25) {};
		\node [style=none] (13) at (1, -1.5) {};
		\node [style=none] (14) at (2.25, -0.5) {};
		\node [style=none] (15) at (3.25, -0.5) {};
		\node [style=none, label={[label distance=-2]0:$f_!$}] (17) at (3.25, -1.25) {};
		\node [style=none] (18) at (4.5, -1.25) {};
		\node [style=none] (30) at (0, 0) {$\coloneqq$};
		\node [style=none] (54) at (4.5, 1.5) {};
		\node [style=short box] (55) at (-1.5, 0) {$\beta$};
		\node [style=none] (56) at (-2, 0.5) {};
		\node [style=none] (62) at (-2, -0.5) {};
		\node [style=none] (63) at (-1, -0.5) {};
		\node [style=none] (64) at (-2, -1.5) {};
		\node [style=none] (66) at (-1, 0.5) {};
		\node [style=none] (69) at (-1, 1.5) {};
		\node [style=none] (70) at (-2, 1.5) {};
		\node [style=none] (72) at (-1, -1.5) {};
		\node [style=none] (78) at (2.25, -1.5) {};
		\node [style=none] (79) at (-2, -2) {\strut$F$};
		\node [style=none] (80) at (-1, -2) {\strut$\rigc$};
		\node [style=none] (81) at (-2, 2) {\strut$\rigc$};
		\node [style=none] (82) at (-1, 2) {\strut$f^*$};
		\node [style=none] (83) at (1, -2) {\strut$F$};
		\node [style=none] (84) at (2.25, -2) {\strut$\rigc$};
		\node [style=none] (85) at (3.25, 2) {\strut$\rigc$};
		\node [style=none] (86) at (4.5, 2) {\strut$f^*$};
	\end{pgfonlayer}
	\begin{pgfonlayer}{edgelayer}
		\draw (5.center) to (1.center);
		\draw (6.center) to (2.center);
		\draw [bend left=90, looseness=1.75] (7.center) to (6.center);
		\draw (7.center) to (13.center);
		\draw (17.center) to (15.center);
		\draw [bend left=90, looseness=1.75] (18.center) to (17.center);
		\draw (18.center) to (54.center);
		\draw (64.center) to (62.center);
		\draw (66.center) to (69.center);
		\draw (70.center) to (56.center);
		\draw (72.center) to (63.center);
		\draw (78.center) to (14.center);
	\end{pgfonlayer}
\end{tikzpicture}
\caption{The natural transformation $\beta \colon \rigc f^* \Rightarrow F \rigc$}\label{fig beta}
\end{figure}

\begin{lemma}\label{mates beta}
The following mate of $\alpha$ is equal to $\beta$.
\[ \begin{tikzcd}
\cdset \rar["\rigc"] \ar[dr, bend right=15, "\id"'] \ar[dr, phantom, bend left=15, "{\scriptstyle\Nearrow} \eta"] & 
\scycop \rar["F"]  \dar["\hNerve"] \ar[dr, phantom, "{\scriptstyle\Nearrow} \alpha"] & 
\soperad \dar["\hNerve"']  \ar[dr, bend left=15, "\id"] \ar[dr, phantom, bend right=15, "{\scriptstyle\Nearrow} \varepsilon"] 
\\
& \cdset \rar["f^*"'] & \dset \rar["\rigc"'] & \soperad
\end{tikzcd} \]
\end{lemma}
\begin{proof}
Figure~\cref{fig mates beta} is a proof using string diagrams.
The first equality uses the definition of $\beta$ in terms of $\kappa$, while the second is the definition of the natural isomorphism $\alpha$.
Since $\alpha$ is the mate of $\beta$, the result follows.
\end{proof}

\begin{figure}
\begin{tikzpicture}
	\begin{pgfonlayer}{nodelayer}
		\node [style=short box] (0) at (4, 3.25) {$\kappa$};
		\node [style=none] (1) at (4.5, 3.75) {};
		\node [style=none] (2) at (3.5, 3.75) {};
		\node [style=none, label={[label distance=-2]180:$\rigc$}] (5) at (4.5, 4.75) {};
		\node [style=none, label={[label distance=-2]180:$L$}] (6) at (3.5, 4.5) {};
		\node [style=none] (7) at (2.25, 4.5) {};
		\node [style=none] (13) at (2.25, 1.75) {};
		\node [style=none] (14) at (3.5, 2.75) {};
		\node [style=none] (15) at (4.5, 2.75) {};
		\node [style=none, label={[label distance=-2]0:$f_!$}] (17) at (4.5, 2) {};
		\node [style=none] (18) at (5.75, 2) {};
		\node [style=none] (30) at (0.25, 3.25) {$=$};
		\node [style=short box] (37) at (9.5, 3.25) {$\alpha$};
		\node [style=none] (38) at (10, 3.75) {};
		\node [style=none] (39) at (9, 3.75) {};
		\node [style=none] (40) at (10, 4.75) {};
		\node [style=none] (41) at (9, 4.75) {};
		\node [style=none] (44) at (9, 2.75) {};
		\node [style=none] (45) at (10, 2.75) {};
		\node [style=none] (46) at (9, 1.75) {};
		\node [style=none] (47) at (10, 1.75) {};
		\node [style=none] (51) at (7.75, 3.25) {$=$};
		\node [style=none] (54) at (5.75, 4.75) {};
		\node [style=short box] (55) at (-2.5, 3.25) {$\beta$};
		\node [style=none] (56) at (-3, 3.75) {};
		\node [style=none] (62) at (-3, 2.75) {};
		\node [style=none] (63) at (-2, 2.75) {};
		\node [style=none] (64) at (-3, 1.75) {};
		\node [style=none] (66) at (-2, 3.75) {};
		\node [style=none] (68) at (-0.75, 4.75) {};
		\node [style=none] (69) at (-2, 4.75) {};
		\node [style=none] (70) at (-3, 4.75) {};
		\node [style=none] (72) at (-2, 1.75) {};
		\node [style=none] (73) at (-0.75, 1.75) {};
		\node [style=none] (74) at (-4.25, 4.75) {};
		\node [style=none] (75) at (-4.25, 1.75) {};
		\node [style=none] (76) at (1.25, 4.75) {};
		\node [style=none] (77) at (1.25, 1.75) {};
		\node [style=none, label={[label distance=-2]0:$\rigc$}] (78) at (3.5, 1.75) {};
		\node [style=none] (79) at (6.75, 1.75) {};
		\node [style=none] (80) at (6.75, 4.75) {};
		\node [style=none] (81) at (1.25, 1.25) {\strut$\hcn$};
		\node [style=none] (82) at (2.25, 1.25) {\strut$F$};
		\node [style=none] (83) at (5.75, 5.25) {\strut$f^*$};
		\node [style=none] (84) at (6.75, 5.25) {\strut$\hcn$};
		\node [style=none] (85) at (9, 1.25) {\strut$\hcn$};
		\node [style=none] (86) at (10, 1.25) {\strut$F$};
		\node [style=none] (87) at (9, 5.25) {\strut$f^*$};
		\node [style=none] (88) at (10, 5.25) {\strut$\hcn$};
		\node [style=none] (89) at (-4.25, 1.25) {\strut$\hcn$};
		\node [style=none] (90) at (-3, 1.25) {\strut$F$};
		\node [style=none] (91) at (-2, 5.25) {\strut$f^*$};
		\node [style=none] (92) at (-0.75, 5.25) {\strut$\hcn$};
	\end{pgfonlayer}
	\begin{pgfonlayer}{edgelayer}
		\draw [style=blue edge] (5.center) to (1.center);
		\draw (6.center) to (2.center);
		\draw [bend left=90, looseness=1.75] (7.center) to (6.center);
		\draw (7.center) to (13.center);
		\draw (17.center) to (15.center);
		\draw [bend left=90, looseness=1.75] (18.center) to (17.center);
		\draw [style=blue edge] (40.center) to (38.center);
		\draw (41.center) to (39.center);
		\draw [style=blue edge] (46.center) to (44.center);
		\draw (47.center) to (45.center);
		\draw (18.center) to (54.center);
		\draw (64.center) to (62.center);
		\draw (66.center) to (69.center);
		\draw [style=blue edge] (70.center) to (56.center);
		\draw [style=blue edge, bend left=90, looseness=1.75] (73.center) to (72.center);
		\draw [style=blue edge] (72.center) to (63.center);
		\draw [style=blue edge] (68.center) to (73.center);
		\draw [style=blue edge] (74.center) to (75.center);
		\draw [style=blue edge, bend left=90, looseness=1.75] (74.center) to (70.center);
		\draw [style=blue edge, bend left=90, looseness=1.75] (76.center) to (5.center);
		\draw [style=blue edge] (76.center) to (77.center);
		\draw [style=blue edge, bend left=90, looseness=1.75] (79.center) to (78.center);
		\draw [style=blue edge] (79.center) to (80.center);
		\draw [style=blue edge] (78.center) to (14.center);
	\end{pgfonlayer}
\end{tikzpicture}
\caption{Proof of \cref{mates beta}}\label{fig mates beta}
\end{figure}

In order to apply \cref{dch56} to establish that $\cdset$ and $\scyc$ are Quillen equivalent, we will show in \cref{beck-chevalley} that the natural isomorphism $\alpha$ satisfies the Beck--Chevalley condition, i.e.~that its mate $\beta$ is also an isomorphism. 
As part of this, we will make use of a natural transformation $\gamma \colon \rigc \Rightarrow F \rigc f_!$, which we define to be the pasting composite of $\kappa \colon L \rigc \Rightarrow \rigc f_!$ and $\eta \colon \id_{\soperad} \Rightarrow FL$, as depicted below.

\[
\begin{tikzcd}
    \dset \ar[r,"\rigc"] \ar[d,"f_!"] & \soperad \ar[dl, phantom, "{\scriptstyle\Swarrow} \kappa" ] \ar[d,swap,"L"]  \ar[dr,bend left=15,"\id"] \ar[dr, phantom, bend right=15, "{\scriptstyle\Swarrow} \eta"] & {} \\
    \cdset \ar[r,"\rigc"] &  \scycop \ar[r,"F"] & \soperad
\end{tikzcd}
\qquad
\begin{tikzpicture}
	\begin{pgfonlayer}{nodelayer}
		\node [style=short box] (37) at (2.75, 0) {$\kappa$};
		\node [style=none] (38) at (3.25, 0.5) {};
		\node [style=none] (39) at (2.25, 0.5) {};
		\node [style=none, label={above:$\rigc$}] (40) at (3.25, 1.5) {};
		\node [style=none, label={[label distance=-2]180:$L$}] (41) at (2.25, 1.25) {};
		\node [style=none] (42) at (1, 1.25) {};
		\node [style=none] (43) at (1, -1.5) {};
		\node [style=none] (44) at (2.25, -0.5) {};
		\node [style=none] (45) at (3.25, -0.5) {};
		\node [style=none] (46) at (2.25, -1.5) {};
		\node [style=none] (47) at (3.25, -1.5) {};
		\node [style=none] (52) at (0, 0) {$\coloneqq$};
		\node [style=short box] (53) at (-1.5, 0) {$\gamma$};
		\node [style=none] (54) at (-1.5, 0.5) {};
		\node [style=none] (55) at (-2, -0.5) {};
		\node [style=none] (56) at (-1.5, -0.5) {};
		\node [style=none] (57) at (-2, -1.5) {};
		\node [style=none, label={above:$\rigc$}] (60) at (-1.5, 1.5) {};
		\node [style=none] (61) at (-1.5, -1.5) {};
		\node [style=none] (62) at (-1, -0.5) {};
		\node [style=none] (63) at (-1, -1.5) {};
		\node [style=none] (64) at (-2, -2) {\strut$F$};
		\node [style=none] (65) at (-1.5, -2) {\strut$\rigc$};
		\node [style=none] (66) at (-1, -2) {\strut$f_!$};
		\node [style=none] (67) at (1, -2) {\strut$F$};
		\node [style=none] (68) at (2.25, -2) {\strut$\rigc$};
		\node [style=none] (69) at (3.25, -2) {\strut$f_!$};
		\node [style=none] (70) at (-1.5, 2) {};
	\end{pgfonlayer}
	\begin{pgfonlayer}{edgelayer}
		\draw (40.center) to (38.center);
		\draw (41.center) to (39.center);
		\draw [bend left=90, looseness=1.75] (42.center) to (41.center);
		\draw (42.center) to (43.center);
		\draw (46.center) to (44.center);
		\draw (47.center) to (45.center);
		\draw (57.center) to (55.center);
		\draw (60.center) to (54.center);
		\draw (61.center) to (56.center);
		\draw (63.center) to (62.center);
	\end{pgfonlayer}
\end{tikzpicture}

\]

To better understand $\gamma$, we explicitly describe its action on a representable dendroidal set $\repo{T,t}$.  
We may characterize $\gamma_{T,t} \colon \rigc (T,t) = \rigc \repo{T,t} \to F \rigc f_!\repo{T,t} = F \rigc T$ as follows:
\begin{itemize}[left=0pt]
    \item On colors, $\gamma_{T,t}$ sends an edge $e$ to the arc of $e$ pointing away from $t$.
    \item For operation spaces, note that a color profile $(\underline{e};e_0)$ bounds a rooted subtree $(S,s) \subseteq (T,t)$ if and only if the associated color profile $(a_0^\dagger,\ua)$ bounds $S$ as an (unrooted) subtree of $T$. 
    Thus $\rigc (T,t)(\underline{e};e_0) = \rigc T(a_0^\dagger, \unda) = F \rigc T(\unda;a_0)$ for all color profiles $(\underline{e};e_0)$. 
    The map $\gamma_{T,t}$ acts as the identity on these operation spaces.
\end{itemize}

As our proof of the Beck--Chevalley condition will involve induction on skeleta, we must also consider the images of boundaries of representable cyclic dendroidal sets under $\rigc$.

\begin{lemma}\label{bdry-rigidification}
    For $T \in \cdcat$ having at least one vertex, the simplicial cyclic operad $\rigc \partial \repu{T} \in \scycop$ is the subobject of $\rigc T$ described as follows:
    \begin{itemize}[left=0pt]
        \item colors of $\rigc \partial \repu{T}$ are the same as those of $\rigc T$, i.e.~arcs of $T$;
        \item if a color profile $\underline{a}$ does not bound $T$, then $\rigc \partial \repu{T}(\underline{a}) = \rigc T(\underline{a})$;
        \item if a color profile $\underline{a}$ bounds $T$, then $\rigc \partial \repu{T}(\underline{a}) = \partial N\Pbd{T} \subset N\Pbd{T} = \rigc T(\underline{a})$.
    \end{itemize}
    For $(T,t) \in \dendcat$ having at least one vertex, the description of $\rigc \partial \repo{T,t} \in \soperad$ is similar to the above, save that color profiles are composed of edges rather than arcs, and the references to bounding of subtrees are interpreted in the rooted sense.
\end{lemma}

\begin{proof}
    The description for the non-cyclic case is well-known; for a proof, see \cite[Example 4.26]{BonventrePereira:RDIO}. 
    For the cyclic case, let $T \in \cdcat$ and choose an arbitrary root $t$ of $T$. 
    Applying \cref{rmk horns colimits} and the natural isomorphism $\kappa$, we have $\rigc \partial \repu{T} \cong \rigc f_! \partial \repo{T,t} \cong L \rigc \partial \repo{T,t}$. 
    To see that $L \rigc \partial \repo{T,t}$ satisfies the given description, we may argue similarly to the proof of \cref{elliott 827}.
\end{proof}

\begin{corollary} \label{gamma-fully-faithful}
   For each rooted tree $(T,t) \in \dendcat$, the maps of simplicial operads $\gamma_{T,t}$ and $\gamma_{\partial \repo{T,t}}$ are fully faithful. \qed
\end{corollary}

\begin{proof}
    This follows from the characterization of $\gamma_{T,t}$ and \cref{bdry-rigidification}, together with the naturality of $\gamma$.
\end{proof}

Having characterized the image of a boundary under rigidification, we next seek to prove an analogue of \cref{boundary horn pushout} in the setting of simplicial operads (\cref{bdry-F-pushout}). In order to do so, we first prove the following general lemma.

\begin{lemma}\label{lem operad pushout}
Consider the following commutative square in $\operad$:
\[ \begin{tikzcd}[column sep=large]
\sum_i O_i \rar{f = \sum_i f_i} \dar["g"'] & \sum_i P_i \dar{g'} \\
Q \rar["f'"'] & R
\end{tikzcd} \]
Suppose the following conditions hold:
\begin{enumerate}[label=\roman*., ref=\roman*]
  \item The morphism $f' \colon Q \to R$ and each component $f_i \colon O_i \to P_i$ is identity-on-colors and faithful.\label{oppushout idcf}
  \item Each component $g_i \colon O_i \to Q$ and each component $g_i' \colon P_i \to R$ is fully-faithful.\label{oppushout incff}
  \item For any profile $(\underline{x}; y)$, if the map $f' \colon Q(\underline{x}; y) \to R(\underline{x}; y)$ is not bijective, then there exists a unique index $i$ and a unique profile $(\underline{a}; b)$ of $P_i$ such that $g_i'(\underline{a}; b) = (\underline{x}; y)$. \label{oppushout unique lifting}
  \item The operations in $R(\underline{x}; y) \setminus Q(\underline{x}; y)$ do not compose with any non-identity operations of $R$.\label{oppushout no compose}
\end{enumerate}
Then the given square is a pushout.
\end{lemma}
\begin{proof}
The square is a pushout on underlying color sets since the horizontal maps are identity-on-colors by \eqref{oppushout idcf}.
Suppose $\ell \colon Q \to Z$ and $h \colon \sum P_i \to Z$ are maps with $hf = \ell g$; we wish to exhibit a unique operad map $k\colon R \to Z$ with $kf' = \ell$ and $kg' =h$.
(Of course $k$ and $\ell$ must agree on colors, since $f'$ is identity-on-colors.)

Let $(\underline{x}; y)$ be a profile on the bottom; we wish to define $k \colon R(\underline{x}; y) \to Z(\ell\underline{x}; \ell y)$.
We consider two cases. In the first case, if $f' \colon Q(\underline{x}; y) \to R(\underline{x}; y)$ is a bijection then the only possibility is $k = \ell (f')^{-1} \colon R(\underline{x}; y) \to Z(\ell\underline{x}; \ell y)$, which is the unique map with $k f' = \ell$.
If $(\underline{a}; b)$ is a profile of $O_i$ with $g_i(\underline{a}; b) = (\underline{x}; y)$, then by \eqref{oppushout incff} all maps in the square
\begin{equation}\label{eq OPQR} \begin{tikzcd}
O_i(\underline{a}, b) \rar{f_i} \dar["g_i"', "\cong"] & P_i(\underline{a}, b) \dar["\cong"']{g'_i} \\
Q(\underline{x}; y) \rar["f'"'] & R(\underline{x}; y)
\end{tikzcd} \end{equation}
are isomorphisms.
At the profile $(\underline{a};b)$ we thus have \[ k g_i' = \ell (f')^{-1} g_i' = \ell (g_i f_i^{-1} (g_i')^{-1}) g_i' = \ell g_i f_i^{-1} = h_i f_i f_i^{-1} = h_i. \]

In the second case, suppose that $f' \colon Q(\underline{x}; y) \to R(\underline{x}; y)$ is \emph{not} a bijection; by \eqref{oppushout unique lifting} there is a unique index $i$ and a unique profile $(\underline{a}, b)$ of $P_i$ so that we have the square \eqref{eq OPQR}.
Since the vertical maps are bijections, this is a pushout square, hence there is a unique map $k \colon R(\underline{x}; y) \to Z(\ell\underline{x}; \ell y)$ such that $k g_i' = h_i$ and $k f' = \ell$.

We now verify that the construction of $k$ above defines a valid operad map; as we have made no choices in our construction, with the definition of $k$ on both colors and operation sets being forced at each step, this will suffice to prove the universal property of the pushout.
Since $f'$ is identity-on-objects, it is immediate that $k$ preserves identities.
Suppose we are given a nontrivial composite $\alpha \circ_j \beta$ in $R$.
By \eqref{oppushout no compose}, both $\alpha$ and $\beta$ are in the image of $f'$, say $\alpha = f'(\alpha')$ and $\beta = f'(\beta')$. 
As 
\[ (k\alpha) \circ_j (k\beta) = (\ell \alpha') \circ_j (\ell \beta') = \ell(\alpha' \circ_j \beta') = kf'(\alpha' \circ_j \beta') =k(\alpha \circ_j \beta),\]
$k$ preserves nontrivial composites, hence all composites.
Finally, one checks that $k$ respects symmetric group actions.
\end{proof}

\begin{lemma} \label{bdry-F-pushout}
    For any $T \in \cdcat$, the following diagram is a pushout:
    \[ 
    \begin{tikzcd}
  \sum\limits_{t\in \perim T} \rigc \partial \repo{T,t} \rar[hook]  \dar & 
  \sum\limits_{t\in \perim T} \rigc (T,t)  \dar \\
  F \rigc \partial \repu{T}  \rar & F \rigc T 
    \end{tikzcd} 
  \]
Here each vertical map acts on the component corresponding to $t \in \perim T$ as the relevant component of $\gamma$.
\end{lemma}

\begin{proof}
If $T \cong \linear{0}$ is an edge, then $\partial \repu{T}$ is $\varnothing$. 
As all functors in the diagram are left adjoints, the above square is a pushout since the left vertical map is $\id_{\varnothing}$ and the right vertical map is the isomorphism $\rigc (\linear{0}, 0) + \rigc (\linear{0}, 0^\dagger) \to F \rigc \linear{0}$.

Suppose $T$ has at least one vertex.
It is enough to show that the square is a pushout after passage to the $n$-simplices in each operation space, and for this we verify the conditions of \cref{lem operad pushout}.
Mostly these are immediate from \cref{bdry-rigidification,gamma-fully-faithful}.
The most important to notice is that if $F \rigc \partial \repu{T} (\underline{a}; a_0)_n$ is not equal to $F \rigc T (\underline{a}; a_0)_n$, then $a_0^\dagger, \underline{a}$ bounds $T$.
This is why \eqref{oppushout unique lifting} and \eqref{oppushout no compose} hold.
\end{proof} 

\begin{lemma}\label{gamma factorization}
The composite
$
  (\beta f_!) \circ (\rigc \eta) \colon \rigc \Rightarrow \rigc f^* f_! \Rightarrow F \rigc f_!
$
is equal to $\gamma$.
\end{lemma}
\begin{proof}
Figure~\ref{fig gamma factorization} is a proof using string diagrams.
The indicated composite is on the left, while the natural transformation on the right is $\gamma$, by definition.
\end{proof}
\begin{figure}
\begin{tikzpicture}
	\begin{pgfonlayer}{nodelayer}
		\node [style=short box] (0) at (2.75, 0) {$\kappa$};
		\node [style=none] (1) at (3.25, 0.5) {};
		\node [style=none] (2) at (2.25, 0.5) {};
		\node [style=none, label={above:$\rigc$}] (5) at (3.25, 2.5) {};
		\node [style=none, label={[label distance=-2]180:$L$}] (6) at (2.25, 1.25) {};
		\node [style=none] (7) at (1, 1.25) {};
		\node [style=none] (13) at (1, -1.5) {};
		\node [style=none] (14) at (2.25, -0.5) {};
		\node [style=none] (15) at (3.25, -0.5) {};
		\node [style=none] (16) at (2.25, -1.5) {};
		\node [style=none, label={[label distance=-2]0:$f_!$}] (17) at (3.25, -1.25) {};
		\node [style=none] (18) at (4.5, -1.25) {};
		\node [style=none] (30) at (0, 0) {$=$};
		\node [style=short box] (37) at (9.5, 0) {$\kappa$};
		\node [style=none] (38) at (10, 0.5) {};
		\node [style=none] (39) at (9, 0.5) {};
		\node [style=none, label={above:$\rigc$}] (40) at (10, 1.5) {};
		\node [style=none, label={[label distance=-2]180:$L$}] (41) at (9, 1.25) {};
		\node [style=none] (42) at (7.75, 1.25) {};
		\node [style=none] (43) at (7.75, -1.5) {};
		\node [style=none] (44) at (9, -0.5) {};
		\node [style=none] (45) at (10, -0.5) {};
		\node [style=none] (46) at (9, -1.5) {};
		\node [style=none] (47) at (10, -1.5) {};
		\node [style=none] (51) at (6.75, 0) {$=$};
		\node [style=none] (52) at (5.75, -1.5) {};
		\node [style=none] (53) at (5.75, 2.5) {};
		\node [style=none] (54) at (4.5, 2.5) {};
		\node [style=short box] (55) at (-2.75, 0) {$\beta$};
		\node [style=none] (56) at (-3.25, 0.5) {};
		\node [style=none, label={above:$\rigc$}] (58) at (-3.25, 2.5) {};
		\node [style=none] (62) at (-3.25, -0.5) {};
		\node [style=none] (63) at (-2.25, -0.5) {};
		\node [style=none] (64) at (-3.25, -1.5) {};
		\node [style=none] (65) at (-2.25, -1.5) {};
		\node [style=none] (66) at (-2.25, 0.5) {};
		\node [style=none] (67) at (-1, -1.5) {};
		\node [style=none] (68) at (-1, 2.5) {};
		\node [style=none] (69) at (-2.25, 2.5) {};
		\node [style=none, label={[label distance=-2]0:$f^*$}] (70) at (4.5, 0.5) {};
		\node [style=none, label={[label distance=-2]0:$f^*$}] (71) at (-2.25, 1.5) {};
		\node [style=none] (72) at (-3.25, -2) {\strut$F$};
		\node [style=none] (73) at (-2.25, -2) {\strut$\rigc$};
		\node [style=none] (74) at (-1, -2) {\strut$f_!$};
		\node [style=none] (75) at (1, -2) {\strut$F$};
		\node [style=none] (76) at (2.25, -2) {\strut$\rigc$};
		\node [style=none] (77) at (5.75, -2) {\strut$f_!$};
		\node [style=none] (78) at (7.75, -2) {\strut$F$};
		\node [style=none] (79) at (9, -2) {\strut$\rigc$};
		\node [style=none] (80) at (10, -2) {\strut$f_!$};
	\end{pgfonlayer}
	\begin{pgfonlayer}{edgelayer}
		\draw (5.center) to (1.center);
		\draw (6.center) to (2.center);
		\draw [bend left=90, looseness=1.75] (7.center) to (6.center);
		\draw (7.center) to (13.center);
		\draw (16.center) to (14.center);
		\draw (17.center) to (15.center);
		\draw [bend left=90, looseness=1.75] (18.center) to (17.center);
		\draw (40.center) to (38.center);
		\draw (41.center) to (39.center);
		\draw [bend left=90, looseness=1.75] (42.center) to (41.center);
		\draw (42.center) to (43.center);
		\draw (46.center) to (44.center);
		\draw (47.center) to (45.center);
		\draw (53.center) to (52.center);
		\draw [bend left=90, looseness=1.75] (54.center) to (53.center);
		\draw (18.center) to (54.center);
		\draw (58.center) to (56.center);
		\draw (64.center) to (62.center);
		\draw (65.center) to (63.center);
		\draw (68.center) to (67.center);
		\draw [bend left=90, looseness=1.75] (69.center) to (68.center);
		\draw (66.center) to (69.center);
	\end{pgfonlayer}
\end{tikzpicture}
\caption{Proof of \cref{gamma factorization}}\label{fig gamma factorization}
\end{figure}

\begin{lemma}\label{beta-pushout}
    For any $T \in \cdcat$, the diagram below is a pushout:
    \[
    \begin{tikzcd}
        \rigc f^* \partial \repu{T} \ar[r] \ar[d,swap,"\beta"] & \rigc f^* \repu{T} \ar[d,"\beta"] \\
        F \rigc \partial \repu{T} \ar[r] & F \rigc T 
    \end{tikzcd}
    \]
\end{lemma}

\begin{proof}
    Consider the following commuting rectangle:
    \[
    \begin{tikzcd}
        \sum\limits_{t \in \perim T} \rigc \partial \repo{T,t} \ar[r] \ar[d,swap,"\rigc \eta"] \ar[dr, phantom, "\ulcorner" very near end] & \sum\limits_{t \in \perim T} \rigc (T,t) \ar[d,"\rigc \eta"] \\
        \rigc f^* \partial \repu{T} \ar[r] \ar[d,swap,"\beta"] & \rigc f^* \repu{T} \ar[d,swap,"\beta"] \\
        F \rigc \partial \repu{T} \ar[r] & F \rigc T 
    \end{tikzcd}
    \]
    The top square is a pushout, as it is obtained by applying the left adjoint $\rigc$ to one of the pushout squares of \cref{boundary horn pushout}. The composite square is also a pushout by \cref{bdry-F-pushout}, as the vertical composites act as $\gamma$ on each component by \cref{gamma factorization}. 
    Thus the bottom square is a pushout by the pasting law for pushouts.
\end{proof}

We now establish the Beck–Chevalley condition.

\begin{proposition}\label{beck-chevalley}
    The natural transformation $\beta \colon \rigc f^* \Rightarrow F \rigc$ is an isomorphism.
\end{proposition}

\begin{proof}
    Since the domain of the functors under consideration is a presheaf category, and $\beta$ respects colimits as a natural transformation between left adjoints, it suffices to show that for any $T \in \cdcat$, the component of $\beta$ at the representable presheaf $\repu{T}$ is an isomorphism. 
    
    Let $n$ denote the number of vertices of $T$; we will proceed by induction on $n$. Suppose, as our induction hypothesis, that for every tree $S$ with fewer than $n$ vertices, the component of $\beta$ at the representable presheaf $\repu{S}$ is an isomorphism. 
    (In the base case $n = 0$ this condition is vacuous.) 
    Then the component of $\beta$ at $\partial \repu{T}$ is an isomorphism; in the case $n = 0$ this holds because $\partial \repu{T} = \varnothing$ and both the domain and codomain of $\beta$ are left adjoints, while in the case $n \geq 1$ it holds by the induction hypothesis and the fact that $\partial \repu{T}$ is $(n{-}1)$-skeletal. 
    The component of $\beta$ at $\repu{T}$ is a pushout of its component at $\partial \repu{T}$ by \cref{beta-pushout}, thus the component at $\repu{T}$ is an isomorphism as well.

Thus we see that the components of $\beta$ at all representable cyclic dendroidal sets are isomorphisms; it follows that all components of $\beta$ are isomorphisms.
\end{proof}

We conclude this section with its main result: the model structures for cyclic quasi-operads and for simplicial cyclic operads are Quillen equivalent.

\begin{theorem}\label{C-nerve-Quillen-equiv}
    The adjunction $\rigc : \cdset \rightleftarrows \scycop : \hNerve$ is a Quillen equivalence. 
\end{theorem}

\begin{proof}
By \cref{mates beta}, the natural transformations $\alpha$ and $\beta$ are mates; they are also natural isomorphisms (using \cref{beck-chevalley} for $\beta$).
\[
  \begin{tikzcd}
    \cdset \rar["\rigc"] \dar["f^*"']  & 
    \scycop \dar["F"] 
    &
    \cdset  \dar["f^*"'] \ar[dr, phantom, "{\alpha \scriptstyle\Searrow} "] & 
    \scycop \dar["F"] \lar["\hcn"']
    \\
    \dset \rar["\rigc"'] \ar[ur, phantom, "{\scriptstyle\Nearrow} \beta"] & 
    \soperad
    &
    \dset  & 
    \soperad \lar["\hcn"]
\end{tikzcd}
\]
Since $\rigc \colon \dset \to \soperad$ is a left Quillen equivalence, we may apply \cref{dch56} to deduce that $\rigc \colon \cdset \to \scycop$ is as well.
\end{proof}

\section{Complete Segal cyclic dendroidal spaces}\label{sec rezk ms}

Recall simplicial presheaves on any dualizable generalized Reedy category (such as $\dendcat$ and $\cdcat$ by \cref{prop gen reedy}) admits the Reedy model structure from \cite{BergerMoerdijk:OENRC}. 

\begin{proposition}\label{upsilon reedy v induced}
The model structure on $\cdspace$ which is right-induced from the Reedy model structure on $\dspace$ along $f^*\colon \cdspace \to \dspace$ is equal to the Reedy model structure.
\end{proposition}
\begin{proof}
Recall that $f\colon \dendcat \to \cdcat$ creates positive, negative, and active maps (\cref{active-inert-def}, \cref{def reedy structure}).
Since $f$ is a discrete fibration by \cref{disc fib}, so too is $\dendcat^+ \to \cdcat^+$.
Similarly, since $\dendcat_\actrm \to \cdcat_\actrm$ is a discrete opfibration by \cref{disc opfib}, so too is $\dendcat^- \to \cdcat^-$.
Then $f^\op \colon \dendcat^\op \to \cdcat^\op$ satisfies the conditions of \cref{nice reedy functor}, and the result follows.
\end{proof}

Without any decoration, the default model structure on $\dspace$ or $\cdspace$ will be the Reedy model structure.
We next turn to another model structure on $\dspace$, called the \emph{dendroidal Rezk model structure}. 

Recall from \cref{ss segal condition} the Segal core inclusion associated to a (rooted) tree.
Consider the following set $S$ of maps of $\dset \subseteq \dspace$:
\begin{itemize}
\item The Segal core inclusions $\segcore{T,t} \to \repo{T,t}$ .
\item The endpoint inclusion $\{0\} \to \freeiso$.
\end{itemize}
As the inclusion $\dset \subseteq \dspace$ sends normal monomorphisms to (Reedy) cofibrations (see \cite[Corollary 12.2]{HeutsMoerdijk:SDHT}), maps in $S$ are cofibrations with cofibrant domains and codomains.
The left Bousfield localization of the Reedy model structure at the set $S$ is called the dendroidal Rezk model structure and is denoted henceforth by $\rezko$.

\begin{theorem}\label{cyclic dendroidal rezk}
The following model structures on the category of cyclic dendroidal spaces exist and are equal:
\begin{itemize}
\item The right-induced model structure from the dendroidal Rezk model structure. 
\item The left Bousfield localization of the Reedy model structure at the set $f_!S$.
\end{itemize}
\end{theorem}
We call this model structure the \mydef{cyclic dendroidal Rezk model structure} and denote it by $\rezku$.
\begin{proof}
Let $S'$ denote the set of maps consisting of $\{0\} \to \freeiso$ and all inner horn inclusions $\Lambda^{T,t}_e \to \repo{T}$.
It follows from \cite[Proposition 5.5]{CisinskiMoerdijk:DSSIO} that the left Bousfield localization of $\dspace$ at $S'$ is equal to $\rezko$.
We calculated in \cref{boundary horn pushout} that the image under $f^*f_!$ of an inner horn inclusion is a pushout of a coproduct of inner horn inclusions, hence is an $S'$-local equivalence.
Meanwhile, $f^*f_!(\{0\} \to \freeiso)$ is isomorphic to a coproduct of two copies of $\{0\} \to \freeiso$, hence is an $S'$-local equivalence.
We may thus apply \cref{localization existence} and \cref{upsilon reedy v induced} to deduce that the right-induced model structure along $f^* \colon \cdspace \to \rezko$ is the left Bousfield localization $\mathscr{L}_{f_!S'} (\cdspace)$ at $f_!S'$.

In particular, $f_! \colon \rezko \to \mathscr{L}_{f_!S'} (\cdspace)$ and $f^* \colon \mathscr{L}_{f_!S'} (\cdspace) \to \rezko$ are left Quillen functors, hence $f^*f_!$ is a left Quillen endofunctor on $\rezko$.
Since the Segal core inclusions are acyclic cofibrations in $\rezko$, they are sent to $S$-local equivalences by $f^*f_!$.
We then apply \cref{localization existence} a second time, this time using the set $S$, to deduce that the right-induced model structure along $f^* \colon \cdspace \to \rezko$ is the left Bousfield localization $\mathscr{L}_{f_!S} (\cdspace)$ at $f_!S$. 
\end{proof}

\begin{remark}
Notice that in the proof of \cref{cyclic dendroidal rezk} we showed that $\rezku$ is the left Bousfield localization of $\cdspace$ at the set of inner horn inclusions, along with $\{0\} \to \freeiso$. 
One could also do all of this without $\{0\} \to \freeiso$ (which imposes completeness) to obtain a model structure for cyclic dendroidal Segal spaces (right-induced from \cite[Definition 5.4]{CisinskiMoerdijk:DSSIO}).
The underlying $\infty$-category of this model structure is the category of Segal $\cdcat^\op$-spaces from \cite[\S2]{ChuHaugseng:HCASC}.
\end{remark}

We next prove that cyclic dendroidal Rezk model structure is Quillen equivalent to the model structure for cyclic quasi-operads.

\begin{theorem}\label{cdset include cdspace}
The inclusion $\cdset \subseteq \rezku$ is a left Quillen equivalence.
\end{theorem}
\begin{proof}
The right adjoint $\tilde B$ to the inclusion $\tilde A \colon \cdset \hookrightarrow \cdspace$ is evaluation at zero: $(\tilde B X)_T = (X_T)_0$.
The counit $\tilde \varepsilon_X \colon \tilde A \tilde B X \to X$ is levelwise given by the inclusion of the constant simplicial set $(X_T)_0$ into $X_T$.
Of course similar considerations hold in the dendroidal setting for $A : \dset \rightleftarrows \dspace : B$ (see proof of Proposition 4.8 of \cite{CisinskiMoerdijk:DSSIO}).
We thus have a map of adjunctions 
\[ \begin{tikzcd}
\cdset \rar[shift left,"\tilde A"] \dar[swap]{f^*} & \cdspace \lar[shift left,"\tilde B"] \dar{f^*} \\
\dset \rar[shift left, "A"] & \dspace \lar[shift left,"B"]
\end{tikzcd} \]
since $f^* \tilde A = A f^*$, $f^* \tilde B = B f^*$, and $f^*\tilde \varepsilon_X = \varepsilon_{f^*X}$.
The result follows by applying \cite[Theorem 5.6]{DrummondColeHackney:CERIMS} (or \cref{dch56}) to the Quillen equivalence established in \cite[Corollary 6.7]{CisinskiMoerdijk:DSSIO}.
\end{proof}

\begin{remark}
There is another Quillen equivalence $\rezko \rightleftarrows \dset$ from \cite[Proposition 6.11]{CisinskiMoerdijk:DSSIO}, where the left Quillen functor goes in the reverse direction (the dendroidal analogue of \cite[Theorem 4.12]{JoyalTierney:QCSS}).
However, this adjunction uses the tensor product of dendroidal sets, which in turn relies on the Boardman--Vogt tensor product of operads.
We do not know a general cyclic analogue of these constructions, so do not attempt to lift this second Quillen equivalence.
It would be interesting to know if this can be done, which would only require one to define the tensor product of an anti-involutive simplicial set with $\repu{T}$.
\end{remark}

\section{Planar (cyclic) \texorpdfstring{$\infty$}{∞}-operads}\label{sec planar}

We now turn to planar operads and planar cyclic operads.
A \emph{planar operad} is just an operad but without the symmetric group actions.
These are the same thing as (small) multicategories.
A \emph{planar cyclic operad} is defined by modifying the definition of cyclic operad so that the $\Sigma_n^+ \cong \Sigma_{n+1}$ actions are replaced with an action by a cyclic subgroup of order $n+1$.
See \cite[\S4.1]{DrummondColeHackney:DKHTCO} for details.
These are the same thing as the (small) cyclic multicategories of \cite[Definition 3.3]{ChengGurskiRiehl:CMMAM} (see \cite[\S4.3]{BataninBerger:LPOHC} for the monochrome case).

There are forgetful functors as in the commutative diagram below, where $\plcyc$ (resp.\ $\plop$) denotes the category of planar cyclic operads (resp.\ planar operads).
\begin{equation}\label{eq diag left adjoints}
\begin{tikzcd}
\cyc \rar \dar{F} & \plcyc \dar \\
\operad \rar & \plop
\end{tikzcd} \end{equation}
All of these functors have left adjoints.
The left adjoint of the right vertical map is a variation of \cref{def cycop adjoint string}.
The left adjoint $\sym \colon \plop \to \operad$ is described explicitly in Sections 20.1 and 20.2 of \cite{Yau:CO}\footnote{More precisely, $\operad \to \plop$ is a fibered right adjoint over $\set$, and \cite{Yau:CO} describes the left adjoints in the fibers.}
and also appears in \cite[\S1.1]{Weiss:DS}.
The left adjoint $\sym \colon \plcyc \to \cyc$ has appeared for example in \cite{Markl:MEOSFT} in the monochrome case, and in the general case is used in \cite[\S5]{Walde:2SSIIO}.

The symmetrization functor induces an equivalence $\plop \simeq \plop_{/\ast} \to \operad_{/\ass}$ where $\ass$ is the symmetrization of the terminal planar operad.
Similarly, the cyclic associative operad $\ass$ (whose underlying operad is the associative operad) is the symmetrization of the terminal planar cyclic operad, and symmetrization induces the equivalence $\plcyc \simeq \cyc_{/\ass}$.

\begin{definition} We define the following two presheaves on $\cdcat$:
\begin{enumerate}[left=0pt]
\item The \mydef{presheaf of roots} $\rooting \in \cdset$ is the image of the terminal operad $\comm$ under $NL \colon \operad \to \cyc \to \cdset$. 
The value of $\rooting$ at a tree $T$ is $\rooting_T = \perim(T)$.
\item The \mydef{presheaf of planar structures} $\planing \in \cdset$ is the image of the terminal planar cyclic operad under $\plcyc \to \cyc \to \cdset$, i.e.\ $\planing = N(\ass)$.
An element of $\planing_T$ is a choice of cyclic order on the set $\perim(\medstar_v)$ for each vertex $v$ or $T$.
\end{enumerate}
\end{definition}

The discrete fibration associated to $\rooting$ is root-elision $f\colon \dendcat \to \cdcat$. 
We let $p \colon \pltree \to \cdcat$ be the discrete fibration associated to $\planing$.
We also the have the product presheaf $\planing \times \rooting$ with associated discrete fibration $q \colon \plrtree \to \cdcat$.
The projections $\planing \times \rooting \to \rooting$ and $\planing \times \rooting \to \planing$ give rise to maps of discrete fibrations $q\to f$ and $q \to p$, exhibiting $q$ as the product in the category of discrete fibrations over $\cdcat$.
In other words, we have a pullback square of discrete fibrations.
\begin{equation}\label{eq diag disc fib square} \begin{tikzcd}
\plrtree \rar{p} \dar[swap]{f} \drar[phantom, "\lrcorner" very near start] & \dendcat \dar{f} \\
\pltree \rar[swap]{p} & \cdcat
\end{tikzcd} \end{equation}
The objects of $\pltree$ are \emph{planar trees} and objects of $\plrtree$ are \emph{planar rooted trees}.
Notice another description of a planar rooted tree: it is a rooted tree together with a total order of $\inp(v)$ for each vertex $v$.
(The dendroidal set corresponding to $p\colon \plrtree \to \dendcat$ is the presheaf of planar structures appearing for instance in Example 3.20(d) of \cite{HeutsMoerdijk:SDHT}, which we also call $\planing$.)
The category of planar trees $\pltree$ was called $\dendcat_{\textup{cyc}}$ in \cite{Walde:2SSIIO}.

These shape categories organize into the following cube, where the front face is the square of left adjoints, the back face is \eqref{eq diag disc fib square}, and the diagonal maps are fully faithful inclusions.
\[ \begin{tikzcd}[sep=tiny]
\plrtree \ar[rr] \ar[dd] \ar[dr] & & \dendcat \ar[dd] \ar[dr] \\
& \plop \ar[rr, crossing over] & & \operad \ar[dd]
\\
\pltree \ar[rr] \ar[dr] & & \cdcat \ar[dr] \\
& \plcyc \ar[from = uu, crossing over] \ar[rr] & & \cyc
\end{tikzcd} \]

Consider the following cube, whose vertical maps arise from restriction along $f$, diagonal maps are given by left Kan extension along $p$, and horizontal maps are levelwise the inclusion of discrete objects.
\begin{equation}\label{cube p presheaf}
\begin{tikzcd}[sep=tiny]
\cpdset \ar[rr] \ar[dd] \ar[dr] & & \cpdspace \ar[dd] \ar[dr] \\
& \cdset \ar[rr, crossing over] & & \cdspace \ar[dd]
\\
\pdset \ar[rr] \ar[dr] & & \pdspace \ar[dr] \\
& \dset \ar[from = uu, crossing over] \ar[rr] & & \dspace
\end{tikzcd} \end{equation}

The categories on the front face of the cube all have model structures we've discussed -- the model structure for (cyclic) quasi-operads on the left and the (cyclic) dendroidal Rezk model structure on the right.

If $\catM$ is a model category and $Z$ is an object, recall the \mydef{slice model structure} on the category $\catM_{/Z}$ of objects over $Z$ \cite[7.6.5]{Hirschhorn:MCL}, where each of the three classes of maps constituting a model structure is created in $\catM$ under the functor $\catM_{/Z} \to \catM$. 
This model structure exists without any additional hypotheses, but if $\catM$ has good properties then so too does $\catM_{/Z}$.
This includes the case when $\catM$ is (left or right) proper, cofibrantly generated, cellular, or simplicial (see \cite[Theorem 15.3.6]{MayPonto:MCAT}, \cite{Hirschhorn:OUCFGMC}, and \cite[Ch.\ II, \S 2, Proposition 6]{Quillen:HA}).

\begin{lemma}\label{lem planar models}
Each of the categories $\pdset$, $\pdspace$, $\cdset$, and $\cpdspace$ appearing on the back face of \eqref{cube p presheaf} admits a model structure whose cofibrations, fibrations, and weak equivalences are created by the diagonal maps.
\end{lemma}
\begin{proof}
We have $\pdset$ is equivalent to the slice category $\dset_{/\planing}$ over the planing presheaf, and $p_! \colon \pdset \to \dset$ is equivalent to the natural forgetful functor $\dset_{/\planing} \to \dset$. 
The indicated model structure is the slice model structure \cite[7.6.5]{Hirschhorn:MCL}.
The other three cases are analogous.
\end{proof}

We give another description of three of these model structures.
\begin{proposition}
The model structure on $\pdset$ is equal to the one from \cite[\S8.2]{Moerdijk:LDS}.
The model structure on $\cpdset$ is right induced along $f^* \colon \cpdset \to \pdset$, and the model structure on $\cpdspace$ is right-induced along $f^* \colon \cpdspace \to \pdspace$.
\end{proposition}
\begin{proof}
The agreement of the slice model structure on $\dset_{/\planing} \simeq \pdset$ with the model structure of Moerdijk was established by Gagna; see Section 4.2 and especially Remark 4.2.10 of \cite{Gagna:Planar}.
To see that $\cpdset \to \pdset$ (resp.\ $\cpdspace \to \pdspace$) create fibrations and weak equivalences, just note that the other three maps in the left (resp.\ right) square of \eqref{cube p presheaf} do so.
\end{proof}

We now provide an alternative description of the model structures on simplicial presheaves.
This is strictly optional: the reader may immediately skip to \cref{planar quasi rezk}.
Let $S$ be the set of maps of $\cdset$ consisting of the Segal core inclusions and the double $\freeiso$-inclusion from \cref{def double J}, and similarly for $S_\planing$ in $\cpdset$.
Similarly, let $S_\rooting$ (resp.\ $S_{\planing \times \rooting}$) be the set of maps of $\dset$ (resp.\ $\pdset$) consisting of the Segal core inclusions and the inclusion $\{0\} \to \freeiso$.
We already know that by forming the appropriate left Bousfield localization of the Reedy model structures we have $\rezko = \mathscr{L}_{S_\rooting} (\dspace)$ and $\rezku = \mathscr{L}_S (\cdspace)$. 
The \mydef{planar (cyclic) dendroidal Rezk model structure} is the left Bousfield localization of the Reedy model structure $\prezko = \mathscr{L}_{S_{\planing \times \rooting}} ( \pdspace)$
(or $\prezku = \mathscr{L}_{S_{\planing}} ( \cpdspace)$ in the cyclic case).
The model structure from \cite[Construction 4.4.5]{Walde:2SSIIO} agrees with $\prezko$.

The following is likely true in greater generality.
However, we are unaware of a reference, and have chosen hypotheses that allow the proof to proceed expeditiously.

\begin{theorem}\label{thm localization vs slice}
Let $\catM$ be a left proper, combinatorial, simplicial model category, and let $S$ be a set of cofibrations between cofibrant objects. 
If $Z$ is an $S$-local object, then there is an equality of model structures
\[
  \mathscr{L}_{S_Z} (\catM_{/Z}) = (\mathscr{L}_S \catM)_{/Z}
\]
on the slice category $\catM_{/Z}$, where $S_Z$ is preimage of $S$ under $\catM_{/Z} \to \catM$.
\end{theorem}
\begin{proof}
Write $\catN \coloneq \catM_{/Z}$ and $F : \catN \rightleftarrows \catM : U$ for the associated adjunction.
The slice category $\catN$ is again a left proper, combinatorial, and simplicial model category. 
For brevity, let $\catN_1 = \mathscr{L}_{S_Z} (\catM_{/Z})$ and $\catN_2 = (\mathscr{L}_S \catM)_{/Z}$; both of these exist by \cite[A.3.7.3]{Lurie:HTT}.
They have the same cofibrations.

Suppose $r\colon W \to Z$ is $S_Z$-local; we aim to show that $r$ is fibrant in $\catN_2$.
As the map $r$ is, in particular, a fibrant object in the slice model structure, it is a fibration in $\catM$.
Let $i\colon A \to B$ be an element of $S$.
We have a commutative square of Kan fibrations between Kan complexes:
\begin{equation}\label{eq mapM square} \begin{tikzcd}
\map_\catM(B,W) \rar{i^*} \dar{r_*}  & \map_\catM(A,W) \dar{r_*} \\
\map_\catM(B,Z) \rar{i^*} & \map_\catM(A,Z).
\end{tikzcd} \end{equation}
For each vertex $u\colon B\to Z$ on the lower left, the corresponding map on fibers is $\map_\catN(u,r) \to \map_\catN(ui,r)$, which is a trivial fibration since $r$ is $S_Z$-local.
Since $Z$ was assumed $S$-local, the bottom $i^*$ of \eqref{eq mapM square} is a trivial fibration, so we conclude the same is true of $i^* \colon \map_\catM(B,W) \to \map_\catM(A,W)$.
Thus $W$ is $S$-local, and $r \colon W\to Z$ is a fibration in $\mathscr{L}_S \catM$ by \cite[3.3.16]{Hirschhorn:MCL}, hence  fibrant in $\catN_2$.

Since $\catN$ is left proper, the fibrant objects of $\catN_1$ are precisely the $S_Z$-local objects (\cite[3.4.1]{Hirschhorn:MCL} or \cref{fibrant object determination}), which we just saw are among the fibrant objects of $\catN_2$. 
The composite $\catN \to \catM \to \mathscr{L}_S \catM$ sends everything in $S_Z$ to a weak equivalence, and $F \colon \catN_2 \to \mathscr{L}_S \catM$ creates weak equivalences, so $\id \colon \catN \to \catN_2$ sends everything in $S_Z$ to a weak equivalence.
By \cite[Proposition 11.24]{HeutsMoerdijk:SDHT}, each fibrant object of $\catN_2$ is fibrant in $\mathscr{L}_{S_Z}(\catN) = \catN_1$.
As the two model structures have the same cofibrations and same fibrant objects, they are equal.
\end{proof}

\begin{corollary}
Each functor in the the following square
\[ 
\begin{tikzcd}
\prezko \rar{p_!} \dar[swap]{f_!} & \rezko \dar{f_!} \\
\prezku \rar[swap]{p_!} & \rezku
\end{tikzcd}
\]
creates weak equivalences, fibrations, and cofibrations.
Consequently, the model structures on the left agree with those in \cref{lem planar models}.
\end{corollary}
\begin{proof}
The corresponding result for the square of Reedy model structures holds by \cref{thm reedy dopfib}.
The result then follows from \cref{thm localization vs slice}, making use of the slice/presheaf equivalence $\spshf{\cdcat_{/Z}} \simeq \cdspace_{/Z}$, noting that $\planing$ and $\rooting$ are both $S$-local.
The only thing to watch out for is that there are \emph{two} maps $f_!(\freeiso) \to \rooting$, not one, but they are isomorphic in $\cdspace_{/\rooting}$ so this does not affect the localization.
\end{proof}

\begin{theorem}\label{planar quasi rezk}
The inclusions  $\pdset \to \prezko$ and $\cpdset \to \prezku$ are left Quillen equivalences.
\end{theorem}
\begin{proof}
Notice that the unit of the $\dset \rightleftarrows \dspace$ adjunction is the identity. 
We write $\planing \in \dset \subseteq \dspace$. 
Since $\planing$ is a discrete dendroidal space, it is Reedy fibrant.
It is local with respect to Segal core inclusions since it is the nerve of the operad $\ass$.
Since the underlying category of $\ass$ is the trivial category, $\planing$ is also local with respect to $\{0\} \to \freeiso$.
Hence $\planing$ is fibrant in $\rezko$.
Since $\dset \to \rezko$ is a left Quillen equivalence, Proposition 3.1(b)(iii) of \cite{Li:NMSC} implies that the induced left adjoint $\dset_{/\planing} \to (\rezko)_{/\planing}$ is a left Quillen equivalence.
This is equivalent to the functor in question being a left Quillen equivalence.
The second statement is analogous, using \cref{cdset include cdspace}.
\end{proof}

Now we turn to planar analogues of the homotopy coherent nerve adjunction between (cyclic) dendroidal sets and simplicial (cyclic) operads.
Notice that the right adjoint $\hcn \colon \soperad \to \dset$ induces a right adjoint functor 
\begin{equation}\label{eq sliced hcn adjunction}
\soperad_{/\ass} \to \dset_{/\hcn(\ass)}
\end{equation}
whose left adjoint applies $\rigc$ and then composes with the counit of the adjunction.
Since $\ass$ is a discrete operad, $\hcn(\ass)$ is $N(\ass)$, i.e.\ $\hcn(\ass) = \planing$.
Since $\operad_{/\ass} \simeq \plop$, the left hand side above is equivalent to the category of simplicial planar operads.
We thus regard \eqref{eq sliced hcn adjunction} as a right adjoint
\begin{equation}\label{splop hcn pdset}
\hcn \colon \splop \to \pdset.
\end{equation}
Likewise, using the cyclic associative operad $\ass$, the cyclic $\hcn$ induces a right adjoint 
$\scyc_{/\ass} \to \cdset_{/\planing}$ which we regard as a functor $\hcn \colon \splcyc \to \cpdset$.

Notice that the definitions of Dwyer--Kan equivalence and isofibration from \cref{def dk and fibrations} can imitated in the evident way for $\splcyc$ and $\splop$.

\begin{lemma}
The slice model structure on $\soperad_{/\ass} \simeq \splop$ (resp.\ $\scyc_{/\ass} \simeq \splcyc$) has weak equivalences the Dwyer--Kan equivalences and fibrations the isofibrations.
\end{lemma}
\begin{proof}
The equivalence $\splop \to \soperad_{/\ass}$ takes $P$ to $O = \sym(P) \to \sym(\ast) =  \ass$. 
We have 
\[
O(a_1, \dots, a_n; a) = \sum_{\sigma \in \Sigma_n} P(a_{\sigma^{-1}(1)}, \dots, a_{\sigma^{-1}(n)}; a).
\]
Thus a map $P \to P'$ is locally a weak homotopy equivalence (resp.\ Kan fibration) if and only if $O \to O'$ is locally a weak homotopy equivalence (resp.\ Kan fibration).
Moreover, the underlying category of $P$ is equal to the underlying category of $O$.
We conclude that $P \to P'$ is a Dwyer--Kan equivalence (resp.\ isofibration) if and only if $O \to O'$ is such.
Since the weak equivalences and fibrations in the slice model structure are created in $\soperad$ the result follows.
The cyclic case is analogous.
\end{proof}

The model structure on $\splop$ was previously considered in Corollary 8.9 of \cite{CisinskiMoerdijk:DSSO}, where it was shown to be proper.
As $F \colon \splcyc \to \splop$ has both adjoints and creates fibrations and weak equivalences, properness of $\splop$ implies properness of $\splcyc$ by \cite[Proposition 2.4]{DrummondColeHackney:CERIMS}. 

\begin{theorem}
The homotopy coherent nerve functor $\hcn \colon \splop \to \pdset$ (resp.\ $\hcn \colon \splcyc \to \cpdset$) is a right Quillen equivalence. 
\end{theorem}
\begin{proof}
Considered as a discrete simplicial operad, $\ass \in \soperad$ is fibrant.
Indeed, $\ass \to \comm$ is locally a Kan fibration since it is a map of discrete simplicial sets, and $\ass(1) \to \comm(1)$ is the identity, hence an isofibration.
Since $\hcn \colon \soperad \to \dset$ is a right Quillen equivalence, Proposition 3.1(b)(ii) of \cite{Li:NMSC} implies that \[ \splop \simeq \soperad_{/\ass} \to \dset_{/\hcn(\ass)} \simeq \pdset\] is a right Quillen equivalence.
The cyclic case is similar, using \cref{C-nerve-Quillen-equiv}.
\end{proof}

\appendix

\section{Lifted model structures}\label{appendix lifted}
In this appendix we record several general facts about induced model structures, including their interaction with standard techniques such as Bousfield localization.
This material is used throughout the paper, but we collect it here as it is of independent interest.

\subsection{Adjoint strings and left-induced model structures}
\label{subsec left induced}

The next lemma appears as part of Theorem 8.2 of \cite{Shulman:UU}; for convenience we include a proof.

\begin{lemma}\label{triple-transfer}
Let $\catM$ and $\catN$ be locally presentable categories, with $\catM$ admitting a cofibrantly generated model structure, and let $F \colon \catN \to \catM$ be a functor having both a left adjoint $L$ and a right adjoint $R$.
Suppose that the composite adjunction $FL \dashv FR$ is Quillen. Then $\catN$ admits model structures right-induced and left-induced by $F$. Moreover, the adjunctions $L \dashv F$ and $F \dashv R$ are Quillen with respect to both of these induced model structures.
\end{lemma}

\begin{proof}
The existence of the right-induced model structure, and the fact that $L \dashv F$ and $F \dashv R$ are Quillen with respect to this model structure, are immediate from \cite[Theorem 2.3]{DrummondColeHackney:CERIMS}. Note that, since $\catN$ is locally presentable by assumption and the right-induced model structure is cofibrantly generated, we may assume that its associated factorizations are functorial.

The existence of the left-induced model structure relies on the Acyclicity Theorem of \cite{HKRS:NSCIMS,GKR:LAMS}.
We verify the assumptions of Theorem~2.2.1 from \cite{HKRS:NSCIMS}, namely:

\begin{enumerate}
    \item for every $X \in \catN$ there exists an object $QX \in \catN$ and a morphism $\varepsilon_X \colon QX \to X$ such that $FQX$ is cofibrant and $F \varepsilon_X$ is a weak equivalence in $\catM$;
    \item for every morphism $f \colon X \to Y$ in $\catN$ there exists a morphism $Qf \colon QX \to QY$ and a commuting square
    \[
    \begin{tikzcd}
        QX \ar[r,"Qf"] \ar[d,swap,"\varepsilon_X"] & QY \ar[d,"\varepsilon_Y"] \\
        X \ar[r,"f"] & Y 
    \end{tikzcd}
    \]
    \item for every $X \in \catN$ there exists a factorization of the co-diagonal map
    \[
    QX + QX \xrightarrow{j} \mathrm{Cyl}(QX) \xrightarrow{p} QX
    \]
    such that $Fj$ is a cofibration and $Fp$ is a weak equivalence in $\catM$.
\end{enumerate}

The functorial factorization in the right-induced model structure induces a cofibrant replacement functor on $\catN$, i.e.\ a functor $Q \colon \catN \to \catN$ equipped with a natural transformation $\varepsilon \colon Q \to \id_{\catN}$ such that for all $X \in \catN$, $QX$ is cofibrant, and $\varepsilon_X \colon QX \to X$ is a weak equivalence, in the right-induced model structure. 
Thus $FQX$ is cofibrant in $\catM$ because $F \dashv R$ is a Quillen adjunction, and $F\varepsilon_X$ is a weak equivalence in $\catM$ by the definition of the weak equivalences in the right-induced model structure. 
Thus we have verified condition (1); condition (2) follows by the functoriality of $Q$ and the naturality of $\varepsilon$. For condition (3), we apply the factorization in the right-induced model structure to factor the co-diagonal map on $QX$ as a composite of a cofibration with a weak equivalence, and proceed by a similar argument.
\end{proof}

\subsection{Lifting Quillen equivalences}\label{sec lifted QE}
The Quillen equivalence lifting theorem from \cite[Theorem 5.6]{DrummondColeHackney:CERIMS} was stated in terms of maps of adjunctions, but holds more generally if written in terms of the Beck--Chevalley condition.
One version of this is the following (see also \cref{rmk more general dch56}).

\begin{theorem}\label{dch56}
Suppose $P \dashv U$ and $\tilde P \dashv \tilde U$ are adjunctions and $\mu, \lambda$ below are natural isomorphisms and are mates of one another (as in \cite[\S2]{KellyStreet:RE2C}).
\[
  \begin{tikzcd}
    \catN \rar["\tilde P"] \dar["F"']  & 
    \catN' \dar["F'"] 
    &
    \catN  \dar["F"'] \ar[dr, phantom, "{\lambda \scriptstyle\Searrow} "] & 
    \catN' \dar["F'"] \lar["\tilde U"']
    \\
    \catM \rar["P"'] \ar[ur, phantom, "{\scriptstyle\Nearrow} \mu"] & 
    \catM'
    &
    \catM  & 
    \catM' \lar["U"]
\end{tikzcd}
\]
Further suppose that $\catM$, $\catM'$ are model categories, the functors $F$ and $F'$ each have both adjoints, and the hypotheses of the first part of \cref{triple-induced-model-structures} hold for each side. 
Give $\catN$, $\catN'$ the right-induced model structures.
If $P \dashv U$ is a Quillen adjunction (resp.\ Quillen equivalence), so is $\tilde P \dashv \tilde U$. 
\end{theorem}

\begin{proof}
If $P \dashv U$ is a Quillen adjunction, then the functor $\tilde U$ is right Quillen because $F\tilde U \cong UF'$ preserves (acyclic) fibrations and $F$ reflects them.

Assume $P\dashv U$ is a Quillen equivalence.
Suppose that $a \in \catN$ is cofibrant, $z\in \catN'$ is fibrant, and $h: a \to \tilde U z$ is adjunct to $h' = \tilde \varepsilon_z \circ \tilde P(h) \colon \tilde P a \to z$.
Our aim is to show that $h$ is a weak equivalence in $\catN$ if and only if $h'$ is a weak equivalence in $\catN'$.
Since $F$ creates weak equivalences, $h$ is a weak equivalence if and only if $Fh \colon Fa \to F \tilde U z$ is. 

Since $\lambda_z$ is an isomorphism, $Fh$ is a weak equivalence if and only if $g = \lambda_z \circ Fh \colon Fa \to U F' z$ is a weak equivalence.
Since $Fa$ is cofibrant, $F'z$ is fibrant, and $U$ is a right Quillen equivalence, the map $g$ is a weak equivalence of $\catM$ if and only if its adjunct $g' = \varepsilon_{F'z} \circ P(g) \colon PFa \to F'z$ of $g$ is a weak equivalence of $\catM'$.

Since $\mu$ and $\lambda$ are mates, we have $(F' \tilde \varepsilon)\circ(\mu \tilde U) = (\varepsilon F')\circ(P \lambda)$, so the right square in the following diagram commutes.
\[ \begin{tikzcd}
PFa \rar{PFh} \dar{\mu_a} & PF \tilde U z \rar{P\lambda_z} \dar["\mu_{\tilde U z}"] & PUF' \dar{\varepsilon_{F'z}} \\
F' \tilde P a \rar{F' \tilde P h} & F' \tilde P \tilde U z \rar{F'\tilde \varepsilon_z} & F'z 
\end{tikzcd} \]
Commutativity of this diagram tells us that $g' = F'(h') \circ \mu_a$; since $\mu_a$ is an isomorphism, we conclude that $g'$ is a weak equivalence if and only if $F'h'$ is a weak equivalence.
Because $F'$ creates weak equivalences, $F'h'$ is a weak equivalence if and only if $h'$ is. 
We conclude that $\tilde P \dashv \tilde U$ is a Quillen equivalence.
\end{proof}

The above proof gives a dual version for left-induced model structures when the hypotheses of the second part of \cref{triple-induced-model-structures} hold on both sides.
This uses that $F'$ is automatically right Quillen (dual to the proof in \cite[Theorem 2.3]{DrummondColeHackney:CERIMS}), so preserves fibrant objects.

\begin{remark}\label{rmk more general dch56}
Several hypotheses in the statement of \cref{dch56} are stronger than what is needed for the proof. 
We never used that $\mu$ is an isomorphism, just that it is a weak equivalence at cofibrant objects.
Assuming $\tilde P \dashv \tilde U$ is a Quillen adjunction, the proof of the second part does not need that $\lambda$ is an isomorphism, just that it is a weak equivalence at fibrant objects.
We also did not use the full hypotheses of \cref{triple-induced-model-structures} in the proof, just that $F,F'$ satisfy the conditions listed in \cite[Theorem 5.6]{DrummondColeHackney:CERIMS}.
\end{remark}

\subsection{Induced model structures and Bousfield localization}

Recall the following basic definitions from \cite{Hirschhorn:MCL}, where $\map(X,Y) \in \sset$ denotes a homotopy function complex for $\catM$.

\begin{definition}\label{def lbl}
Suppose $\cC$ is a class of maps in a model category $\catM$.
\begin{itemize}[left=0pt]
\item 
An object $W\in \catM$ is called \emph{$\cC$-local} if 
\begin{itemize}
  \item it is fibrant in $\catM$, and
  \item $\map(f, W) \colon \map(B, W) \to \map(A,W)$ is a weak equivalence of simplicial sets for every $f\colon A \to B$ in $\cC$.
\end{itemize}
\item A map $f\colon X \to Y$ is a \emph{$\cC$-local equivalence} if $\map(f,W) \colon \map(Y,W) \to \map(X,W)$ is a weak equivalence for every $\cC$-local object $W$.
\end{itemize}
Another model structure $\catM'$ on the same underlying category is called the \emph{left Bousfield localization of $\catM$ at $\cC$} if it has the same cofibrations as $\catM$ and the following condition holds:
\begin{enumerate}[label=(\alph*), ref=\alph*]
  \item The weak equivalences in $\catM'$ are the $\cC$-local equivalences of $\catM$.\label{item lbl le}
\end{enumerate}
If $\catM'$ is the left Bousfield localization of $\catM$ at $\cC$, we denote it by $\mathscr{L}_{\cC}\catM$.
\end{definition}

\begin{lemma}\label{fibrant object determination}
Suppose $\catM$ and $\catM'$ are two model structures on the same underlying category, which have the same class of cofibrations, and let $\cC$ be a class of maps in $\catM$.
Consider the following condition 
\begin{enumerate}[label=(\alph*), ref=\alph*, start=2]
  \item The fibrant objects in $\catM'$ are the $\cC$-local objects of $\catM$.\label{item lbl lo}
\end{enumerate}
Condition \eqref{item lbl lo} implies condition \eqref{item lbl le}. 
If $\catM$ is left proper, then \eqref{item lbl le} implies \eqref{item lbl lo}.
\end{lemma}

We use \cite[Proposition E.1.10]{quadern45}, reproduced below, in our proof.

\begin{proposition}\label{joyal prop}
Suppose $\catM$ and $\catM'$ are two model structures on a category which have the same cofibrations.
Each weak equivalence of $\catM$ is a weak equivalence of $\catM'$ if and only if each fibrant object of $\catM'$ is a fibrant object of $\catM$.
Thus a model structure on a category is determined by its cofibrations and fibrant objects. \qed
\end{proposition}

\begin{proof}[Proof of \cref{fibrant object determination}]
If \eqref{item lbl le} holds, then $\catM'$ is the left Bousfield localization of $\catM$ at $\cC$.
If $\catM$ is left proper, then \eqref{item lbl lo} holds by \cite[3.4.1]{Hirschhorn:MCL}.

Assume \eqref{item lbl lo} holds.
Since $\cC$-local objects are in particular fibrant in $\catM$, we know that the weak equivalences of $\catM$ are also weak equivalences of $\catM'$ by \cref{joyal prop}.
It follows that Reedy weak equivalences in $\catM^\Delta$ are also Reedy weak equivalences in $(\catM')^\Delta$.
Of course Reedy cofibrations coincide in both of these categories.
If $X$ is an arbitrary object, then a Reedy cofibrant approximation 
\[ \begin{tikzcd}
\varnothing \rar[tail] & \mathbf{\widetilde X} \rar["\simeq"] & \text{const} X
\end{tikzcd} \]
in $\catM^\Delta$ is also a Reedy cofibrant approximation in $(\catM')^\Delta$.
Now if $W\in \catM'$ is fibrant (and $X$ is arbitrary) we can use the left homotopy function complex \cite[17.1.1]{Hirschhorn:MCL}
\[
  \map(X,W) = \hom(\mathbf{\widetilde X}, W)
\]
as our mapping space, without worrying about which of the two model structures we are using.
A map $g \colon X \to Y$ is a weak equivalence in $\catM'$ just when $\map(Y,W) \to \map(X,W)$ is a weak equivalence for every fibrant $W\in \catM'$ by \cite[17.7.7]{Hirschhorn:MCL}, but this is precisely the condition that $g$ is a $\cC$-local equivalence.
Thus \eqref{item lbl le} holds.
\end{proof}

\begin{theorem}\label{main localization and lifting theorem}
Suppose that $\catM$ is a left proper model category, $\catN$ is a bicomplete category, and $F \colon \catN \to \catM$ is a functor with left adjoint $L$.
Let $\cC$ be a class of maps in $\catM$, and assume that the left Bousfield localization $\mathscr{L}_{\cC}\catM$ exists, and that the following two model structures on $\catN$ exist:
\begin{itemize}
\item The right-induced model structure $\catN_1$ along $F \colon \catN \to \catM$.
\item The right-induced model structure $\catN_2$ along $F\colon \catN \to \mathscr{L}_{\cC} \catM$.
\end{itemize}
Then $\catN_2$ is the left Bousfield localization of $\catN_1$ at the class of maps $L\cC$.
\end{theorem}
\begin{proof}
We write $\catM_1$ for the original model structure $\catM$, and $\catM_2 = \mathscr{L}_{\cC}\catM$ for the localized model structure.
Since the cofibrations of $\catM_1$ and $\catM_2$ coincide, so too do the acyclic fibrations.
Since $\catN_1$ and $\catN_2$ are right-induced model structures, their acyclic fibrations are precisely the maps which are sent by $F$ to acyclic fibrations in the base model categories $\catM_1$ and $\catM_2$. 
Since those coincide, the acyclic fibrations in $\catN_1$ and $\catN_2$ coincide; so too must the cofibrations.

We now turn to fibrant objects.
An object $W$ of $\catN_2$ is fibrant if and only if $FW \in \catM_2$ is fibrant.
This occurs if and only if $FW$ is fibrant in $\catM_1$ and is $\cC$-local.
We next observe that $FW \in \catM_1$ is $\cC$-local if and only if $W\in \catN_1$ is $L\cC$-local: If $A \to B$ is in $\cC$, we have by \cite[17.4.16]{Hirschhorn:MCL} (see also 17.4.2) the following commutative square with vertical maps weak equivalences
\[ \begin{tikzcd}
\map_{\catM}(B,FW) \rar \dar{\simeq} & \map_{\catM}(A,FW) \dar{\simeq} \\ 
\map_{\catN_1}(LB,W) \rar & \map_{\catN_1}(LA,W)
\end{tikzcd} \]
so the top map is a weak equivalence just when the bottom map is a weak equivalence.
We conclude that $W\in \catN_2$ is fibrant if and only if it is fibrant as an object of $\catN_1$ and is $L\cC$-local.
The result follows from \cref{fibrant object determination}.
\end{proof}

\begin{corollary}\label{localization existence}
Let $\catM$ be a left proper and cofibrantly generated model category and $\catN$ a bicomplete category.
Let $F\colon \catN \to \catM$ be a functor having a left adjoint $L$ and a right adjoint $R$, such that $FL \colon \catM \to \catM$ is left Quillen.
Let $S$ be a set of maps such that the left Bousfield localization $\mathscr{L}_S\catM$ exists and is cofibrantly generated (e.g.\ if $\catM$ is combinatorial \cite[Theorem 4.7]{Barwick:OLRMCLRBL} or cellular \cite[4.1.1]{Hirschhorn:MCL}).
If 
\begin{enumerate}
\item $FL$ takes elements of $S$ to $S$-local equivalences, and \label{assumption S-local}
\item domains and codomains of elements of $S$ are cofibrant in $\catM$, \label{assumption cofibrant}
\end{enumerate}
then we have the following commutative square 
\[
\begin{tikzcd}
\catN_r \rar \dar & \mathscr{L}_{LS} \catN_r \dar \\
\catM \rar & \mathscr{L}_S \catM
\end{tikzcd}
\]
with the model structures in the top row right-induced from those in the bottom row and with horizontal maps left Bousfield localizations. 
\end{corollary}
\begin{proof}
By \cref{triple-induced-model-structures} we obtain a model structure $\catN_1 = \catN_r$ right-induced from the model structure on $\catM$.
As we have \eqref{assumption cofibrant}, each $s\in S$ is a cofibrant approximation to itself. 
Then \eqref{assumption S-local} along with 3.3.18 and 8.5.10 of \cite{Hirschhorn:MCL}
applied to the left Quillen functor
\[ \begin{tikzcd}
\catM \rar{FL} & \catM  \rar{\id} & \mathscr{L}_S\catM.
\end{tikzcd} \]
give that $FL \colon \mathscr{L}_S\catM \to \mathscr{L}_S\catM$ is a left Quillen endo-functor.
Hence we have the model structure $\catN_2$ which is right-induced from the model structure on $\mathscr{L}_S\catM$, again by \cref{triple-induced-model-structures}.
\Cref{main localization and lifting theorem} implies that $\catN_2$ is the left Bousfield localization $\catN_2 = \mathscr{L}_{LS}\catN_1$.
\end{proof}

\subsection{Right-induced and Reedy model structures}

Let $\reedyR$ be a generalized Reedy category in the sense of \cite{BergerMoerdijk:OENRC}.
We write $\reedyR^+$ and $\reedyR^-$ for the positive and negative subcategories. 
We let $\reedyR^+((n))$ for the full subcategory of arrow category of $\reedyR^+$ whose objects are the non-invertible (positive) morphisms whose codomain has degree $n$, and dually for $\reedyR^-((n))$. 
We let $\reedyG_n(\reedyR) \subset \iso(\reedyR)$ be the groupoid of objects of degree $n$. 

For each $n$, we have functors
\[ \begin{tikzcd}[row sep=0]
\reedyR & \reedyR^+((n)) \lar["\dom"'] \rar["\cod"] & \reedyG_n(\reedyR) 
\\
\reedyR & \reedyR^-((n)) \lar["\cod"'] \rar["\dom"] & \reedyG_n(\reedyR) 
\end{tikzcd} \]
as well as the inclusion $j_n \colon \reedyG_n(\reedyR) \to \reedyR$.
For $X\in \catM^{\reedyR}$, we write $X_n$ for the object $j_n^*X\in \catM^{\reedyG_n(\reedyR)}$.
Given a diagram $X$, we also have the latching object $L_n X = \cod_! \dom^* X$ in $\catM^{\reedyG_n(\reedyR)}$, as well as the matching object $M_n X = \dom_* \cod^* X$. 
Recall the natural latching and matching maps $L_n X \to X_n \to M_n X$ (see \cite[\S4]{BergerMoerdijk:OENRC}).
If $f \colon X \to Y$ is a morphism of $\catM^{\reedyR}$, we have the (relative) latching and matching maps
\begin{align*}
    \ell_n(f) &\colon X_n \amalg_{L_n X} L_n Y \to Y_n \\
    m_n (f) &\colon X_n \to M_n X \times_{M_n Y} Y_n
\end{align*}
in $\catM^{\reedyG_n(\reedyR)}$.
By definition, $f$ is a cofibration if and only if $\ell_n(f)$ is a cofibration in the projective model structure on $\catM^{\reedyG_n(\reedyR)}$ for each $n$.
The map $f$ is a fibration if and only if $m_n(f)$ is a fibration in the projective model structure on $\catM^{\reedyG_n(\reedyR)}$ for each $n$.

\begin{lemma}\label{lem proj discfib}
Let $\catM$ be a cofibrantly generated model category.
If $\phi \colon C \to D$ is a discrete fibration, then $\phi^* \colon \catM^D \to \catM^C$ is a left Quillen functor between the projective model structures.
\end{lemma}
(Notice that $\phi^*$ is automatically right Quillen from the definition of the projective model structure.)
\begin{proof}
We show that the right Kan extension functor $\phi_* \colon \catM^C \to \catM^D$ is right Quillen.
As right Kan along a fibration is given by the limit over the fibers (see e.g.\ \cite[\S3]{BergerMoerdijk:OENRC}), and the fibers of $\phi$ are discrete, we have $(\phi_*X)(d) = \prod_{c \in \phi^{-1}(d)} X(c)$.
Fibrations and acyclic fibrations in $\catM$ are closed under products, so if $X \to Y$ is an (acyclic) fibration in $\catM^C$, then $\phi_*(X\to Y)$ is an (acylic) fibration in $\catM^D$.
\end{proof}

Since fibrations and weak equivalences in the projective model structure are levelwise, we have the following:

\begin{lemma}\label{lem proj esssurj}
If $\phi \colon C \to D$ is essentially surjective, then the map $\phi^* \colon \catM^D \to \catM^C$ between projective model structures reflects fibrations and weak equivalences. \qed
\end{lemma}

\begin{lemma}\label{lem latching iso}
Let $\phi \colon \reedyS \to \reedyR$ be a Reedy functor with $\reedyS^+ \to \reedyR^+$ a discrete fibration (resp.\ $\reedyS^- \to \reedyR^-$ a discrete opfibration). 
Then the following square is a pullback.
\[
\begin{tikzcd}
\reedyS^+((n)) \rar{\cod}  \dar{\phi_n^+} & \reedyG_n(\reedyS) \dar{\phi_n} \\
\reedyR^+((n)) \rar{\cod} & \reedyG_n(\reedyR) 
\end{tikzcd}
\qquad \left(\text{resp.\ } 
\begin{tikzcd}
\reedyS^-((n)) \rar{\dom}  \dar{\phi_n^-} & \reedyG_n(\reedyS) \dar{\phi_n} \\
\reedyR^-((n)) \rar{\dom} & \reedyG_n(\reedyR) 
\end{tikzcd}
\right)
\]
In particular, $L_n \phi^* \Rightarrow \phi_n^* L_n$ (resp.\ $\phi_n^* M_n \Rightarrow M_n \phi^*$) is an isomorphism.
\end{lemma}
\begin{proof}
A straightforward check shows that the square in question is a pullback of categories, and the isomorphism $L_n \phi^* \cong \phi_n^* L_n$ is \cite[Lemma 4.4]{BergerMoerdijk:OENRC} (resp.\ \cite[Lemma 4.7]{BergerMoerdijk:OENRC} for $\phi_n^* M_n \cong M_n \phi^*$).
\end{proof}

\begin{lemma}\label{lem lmatching phi}
Let $\phi \colon \reedyS \to \reedyR$ be a Reedy functor, and $f \colon X \to Y$ a map in $\catM^{\reedyS}$.
\begin{enumerate}[left=0pt]
    \item If $\reedyS^+ \to \reedyR^+$ is a discrete fibration, then $\ell_n(\phi^*(f)) \cong \phi_n^* \ell_n(f)$.
    \item If $\reedyS^- \to \reedyR^-$ is a discrete opfibration, then $m_n(\phi^*(f)) \cong \phi_n^* m_n(f)$.
\end{enumerate}
\end{lemma}
\begin{proof}
Let us prove (1).
We have the following commutative diagram, using \cref{lem latching iso} for the isomorphism in the upper left.
\[
\begin{tikzcd}
j_n^* \phi^* X \amalg_{L_n \phi^* X} L_n \phi^*Y \rar \dar{\cong} & j_n^* \phi^* Y \dar{=}
\\
\phi_n^* j_n^* X \amalg_{\phi_n^* L_n X} \phi_n^* L_n Y \dar{\cong} & \phi_n^* j_n^* Y \dar{=}
\\
\phi_n^*(j_n^* X \amalg_{L_n  X} L_n Y ) \rar & \phi_n^* (Y_n)
\end{tikzcd}
\]
The top map is $\ell_n(\phi^*(f))$, and the bottom is $\phi_n^*(\ell_n(f))$, hence these maps are isomorphic.
Case (2) is similar. 
\end{proof}

\begin{proposition}\label{nice reedy functor}
Suppose $\phi \colon \reedyS \to \reedyR$ is a Reedy functor and $\catM$ is a cofibrantly generated model category.
Assume
\begin{enumerate}
    \item $\reedyS^+ \to \reedyR^+$ is a discrete fibration, and\label{pos fib}
    \item $\reedyS^- \to \reedyR^-$ is a discrete opfibration.\label{neg opfib}
\end{enumerate}
Then $\phi^* \colon \catM^{\reedyR} \to \catM^{\reedyS}$ is both left and right Quillen.
Further, if $\phi$ is essentially surjective then the Reedy model structure on $\catM^{\reedyR}$ is right-induced along $\phi^*$ from the Reedy model structure on $\catM^{\reedyS}$.
\end{proposition}
\begin{proof}
Since weak equivalences are levelwise, $\phi^*$ preserves weak equivalences (with no hypothesis on $\phi$).
If $\phi$ is essentially surjective, it also reflects them.
It thus suffices to show that $\phi^*$ preserves fibrations and cofibrations, and, if we assume $\phi$ is essentially surjective, that it reflects fibrations.
Let $f\colon X \to Y$ be a map in $\catM^{\reedyS}$.

Suppose $f\colon X \to Y$ is a fibration in $\catM^{\reedyS}$, i.e.\ $m_n(f)$ is a fibration in $\catM^{\reedyG_n(\reedyS)}$ for all $n$. 
As $\phi_n^* \colon \catM^{\reedyG_n(\reedyS)} \to\catM^{\reedyG_n(\reedyR)}$ preserves fibrations, $\phi_n^* m_n(f)$ is a fibration.
By \cref{lem lmatching phi}, we know $m_n(\phi^*f) \cong \phi_n^* m_n(f)$.
Thus $m_n(\phi_* f)$ is a fibration for all $n$, and we conclude that $\phi_*f$ is a fibration.

Suppose $f\colon X \to Y$ is a cofibration in $\catM^{\reedyS}$, i.e.\ $\ell_n(f)$ is a cofibration in $\catM^{\reedyG_n(\reedyS)}$ for all $n$. 
Since $\reedyS^+ \to \reedyR^+$ is a discrete fibration, so too is $\reedyG_n(\reedyS) \to \reedyG_n(\reedyR)$.
By \cref{lem proj discfib}, $\phi_n^*\ell_n(f)$ is a cofibration. 
Then \cref{lem lmatching phi} tells us $\ell_n(\phi^*(f)) \cong \phi_n^* \ell_n(f)$, so that $\ell_n(\phi^*(f))$ is a cofibration as well. Hence $\phi^*(f)$ is a cofibration.

Finally, assume that $\phi$ is essentially surjective and $f\colon X\to Y$ has the property that $\phi^*(f)$ is a fibration.
We wish to show that $f$ is a fibration.
Again by \cref{lem lmatching phi}, we have $m_n(\phi^*f) \cong \phi_n^* m_n(f)$.
Further, $\phi_n^* \colon \catM^{\reedyG_n(\reedyS)} \to\catM^{\reedyG_n(\reedyR)}$ reflects fibrations 
by \cref{lem proj discfib} (since $\phi_n \colon \reedyG_n(\reedyS) \to \reedyG_n(\reedyR)$ is essentially surjective), so we conclude that $m_n(f)$ is a fibration for all $n$.
\end{proof}

\begin{theorem}\label{thm reedy dopfib}
Suppose $\phi \colon \reedyS \to \reedyR$ is a Reedy functor which is a discrete opfibration and reflects positive (resp.\ negative) maps.
Then $\phi_! \colon \sset^{\reedyS} \to \sset^{\reedyR}$ creates fibrations, cofibrations, and weak equivalences.
\end{theorem}
\begin{proof}
Write $Z \colon \reedyS \to \set$ for the functor associated to the discrete opfibration $\phi$.
Then $\phi_!$ is equivalent to $\sset^{\reedyS} \simeq (\sset^{\reedyR})_{/Z} \to \sset^{\reedyR}$.
The category $(\sset^{\reedyR})_{/Z}$ admits the slice model structure \cite[7.6.5]{Hirschhorn:MCL}, where fibrations, cofibrations, and weak equivalences are created in $\sset^{\reedyR}$.
It thus suffices to show that $\phi_!$ creates two of the three classes of maps.

It is not hard to see that $\phi_!$ creates levelwise weak equivalences.
Indeed, we have \[ (\phi_!X)(r) = \sum_{s \in \phi^{-1}(r)} X(s)\]
since $\phi$ is a discrete opfibration, 
and a coproduct of maps in $\sset$ is a weak equivalence if and only if each of its components is a weak equivalence.

Next consider the special case where $\phi$ is a map of groupoids, all of whose objects have some fixed degree.
Then the Reedy model structure is just the projective model structure, and $\phi_!$ creates levelwise fibrations.
This is because a coproduct of maps in $\sset$ is a fibration if and only if each of its components is a fibration.
Since $\phi_!$ creates weak equivalences and fibrations, it also creates cofibrations as mentioned in the first paragraph.

We now return to the general case.
We will show that $\phi_!$ creates cofibrations.
Since $\phi$ is a discrete opfibration, we have pullback squares 
\[
\begin{tikzcd}
\reedyS^+((n)) \rar{\dom} \dar["\phi"'] \drar[phantom, "\lrcorner" very near start] & \reedyS\dar["\phi"]
& \reedyG_n(\reedyS) \rar{j_n} \dar["\phi_n"']\drar[phantom, "\lrcorner" very near start] & \reedyS\dar["\phi"]
\\
\reedyR^+((n)) \rar[swap]{\dom} & \reedyR
& \reedyG_n(\reedyR) \rar[swap]{j_n} & \reedyR.
\end{tikzcd}
\]
Then by the Beck--Chevalley condition (see \cite[\S3]{BergerMoerdijk:OENRC}) we have $\dom^*\phi_! \cong \phi_!\dom^*$ and $(\phi_n)_! j_n^* \cong j_n^* \phi_!$.
The first of these gives an isomorphism \[ L_n \phi_! = \cod_! \dom^* \phi_! \cong \cod_! \phi_! \dom^* \cong (\phi_n)_! \cod_! \dom^* = (\phi_n)_! L_n \] 
We can use this to deduce that for any map $g \colon X \to Y$ in $\sset^{\reedyS}$ we have $(\phi_n)_! \ell_n(g)$ is isomorphic to $\ell_n(\phi_! g)$, as in the below commutative diagram:
\[
\begin{tikzcd}
(\phi_n)_! \left( X_n \amalg_{L_n X} L_n Y \right)  \rar["({\phi_n)_!(\ell_n(g))}"] \dar{\cong} & (\phi_n)_! Y_n \dar{=}
\\
(\phi_n)_! j_n^* X \amalg_{(\phi_n)_! L_n X} (\phi_n)_! L_n Y \dar{\cong} & (\phi_n)_! j_n^* Y \dar{\cong}
\\
j_n^* \phi_! X \amalg_{ L_n \phi_! X}  L_n \phi_! Y \rar["\ell_n (\phi_! g)"] &  j_n^* \phi_! Y.
\end{tikzcd}
\]
By the second paragraph, $(\phi_n)_!$ creates cofibrations, hence $\ell_n(g)$ is a cofibration in $\sset^{\reedyG_n(\reedyS)}$ if and only if $\ell_n(\phi_!g)$ is a cofibration in $\sset^{\reedyG_n(\reedyR)}$.
Thus $g$ is a Reedy cofibration in $\sset^{\reedyS}$ if and only if $\phi_! (g)$ is a Reedy cofibration in $\sset^{\reedyR}$.
\end{proof}

\bibliographystyle{amsalpha}
\bibliography{cyc_refs}
\end{document}